\documentclass[smallcondensed]{svjour3}     
\usepackage[utf8]{inputenc}
\usepackage{amssymb, amsmath}
\usepackage[english]{babel}
\usepackage{hyperref}
\usepackage{xcolor}

\allowdisplaybreaks

\newcommand{\F}{\mathbb{F}}

\newcommand{\R}{\mathbb{R}}

\newcommand{\Image}{\operatorname{Im}}
\DeclareMathOperator{\Tr}{Tr}

\smartqed  

\begin{document}

\title{Image sets of perfectly nonlinear maps}


\author{Lukas K\"olsch        \and
        Bj\"orn Kriepke \and \\
				Gohar M. Kyureghyan
}


\institute{{Lukas K\"olsch, Bj\"orn Kriepke, Gohar M. Kyureghyan} \at
							University of Rostock, Institute of Mathematics
              Ulmenstrasse 69, 18057 Rostock \\
              \email{\{lukas.koelsch,bjoern.kriepke,gohar.kyureghyan\}@uni-rostock.de}           
}

\date{Received: date / Accepted: date}

\maketitle

\begin{abstract}
We consider  image sets  of  differentially $d$-uniform maps
of finite fields. We present a lower bound on the image size of such maps and
study their preimage distribution, by extending
methods used for planar maps. 
We apply the results to study $d$-uniform Dembowski-Ostrom polynomials.
Further, we focus on a particularly interesting case  of APN maps on binary fields $\F_{2^n}$. 
We show that APN maps with the minimal image size must have a very special preimage distribution. 
We prove that for an even $n$ the image sets of several well-studied families of APN maps are minimal.
We present results connecting the image sets of special maps with their Walsh spectrum. 
Especially, we show that the fact that several  large classes of APN maps have the classical Walsh spectrum
is explained by the minimality of their image sets.
Finally, we present upper bounds on the image size of APN maps.
 For  a non-bijective almost bent map $f$, these results imply
 $\frac{2^n+1}{3}+1 \leq |\Image(f)| \leq 2^n-2^{(n-1)/2}$.

\keywords{Image set \and APN map \and differential uniformity \and Walsh spectrum \and quadratic map 
\and Dembowski-Ostrom polynomial \and plateaued function \and preimage distribution}
\end{abstract}
%
%
%
%
%

\section{Introduction}

Let $p$ be a prime and $q=p^n$. A map $f:\F_q \to \F_q$ is called differentially $d$-uniform (abbreviated $d$-uniform), 
if
\[
	d=\max_{a\neq 0,b\in\F_q} |\{x\in\F_q: f(x+a)-f(x)=b\}|.
\]
A 1-uniform map $f:\F_q \to \F_q$ is called planar, that is $f$ is planar if $f(x+a)-f(x)$ is a permutation for any  $a\in\F_q^*$. Planar maps exist if and only if  $q$ is odd. A map $f$ is called almost perfect nonlinear (APN) if  $f$ is 2-uniform. 
Observe that if $q$ is even, then an equation $f(x+a)+f(x)=b$ has always an even number of solutions,
since $x$ solves it if and only if $x+a$ does so. 
In particular  there are no 1-uniform
maps for $q$ even, and the APN maps have the smallest possible uniformity on binary fields. APN maps and more generally maps in characteristic $2$ with low uniformity are an important research object in cryptography, mainly because they provide good resistance to differential attacks when used as an S-box of a block cipher. For a thorough survey detailing the importance of such maps for cryptography, we refer to~\cite{nybergsurvey}. Moreover, maps with low uniformity are intimately connected to certain codes~\cite{cczpaper,carletdingcodes}. Planar maps can be used for the construction of various structures in combinatorics and algebra, for example difference sets, projective planes and semifields~\cite{pottsurvey}. \\

\noindent
A celebrated result of Ding and Yuan obtained in \cite{ding2006diffset} shows that  image sets of planar maps yield  
skew  Hadamard difference sets which are inequivalent to the Paley–Hada\-mard difference sets. This  disproved  a longstanding conjecture on the
classification of skew Hadamard difference sets and  motivated an interest to a better understanding of
image sets of planar maps, see for example \cite{coulter2011number,kyureghyan2008some,weng2007pseudo}.
The image sets of  $d$-uniform maps with $d>1$  can be used to construct optimal combinatorial objects 
too, as shown in \cite{carlet2016quadratic}. However  the case $d>1$ is less studied compared to $d=1$,
although also $d=1$ is far from being completely understood too. 
Here we extend some of results on the image sets of planar maps with $d=1$ to cover a general $d$. 
The behavior of the image sets of $d$-uniform maps and  proofs  are more
complex  for $d>1$. This is simply explained by the fact that the preimage distribution
of a difference map $f_a: x\mapsto f(x+a)-f(x)$ is not unique and more difficult to control
when $d>1$. The smaller  values  $d$ are easier to handle.\\

\noindent
In this paper we obtain a lower bound for the size of $d$-uniform maps, which is sharp for several classes of $d$-uniform maps. However there are cases, where
we expect  that our bound can be improved. We prove several results on the preimage distribution of $d$-uniform maps.
We observe that some classes of $d$-uniform Dembowski-Ostrom polynomials are uniquely characterized by the size of their image set.
Further we consider in more detail the case $d=2$, that is APN maps, on binary fields. For an APN map $f:\F_{2^n} \to \F_{2^n}$
the lower bound is
\[
		|\Image(f)|\geq 
		\begin{cases}
			\frac{2^n+1}{3} & n \text{ is odd,} \\
			\frac{2^n+2}{3} & n \text{ is even}.
		\end{cases}
	\]
The first published proof for this bound appears  in \cite[Lemma 5]{carlet-heuser-picek}, where a lower bound on the differential uniformity via image set size is presented. 
 Since the study of image sets of APN maps
was not a goal of \cite{carlet-heuser-picek},  the lower bound in it remained unnoticed by most of researchers on APN maps. A systematic
study of the image sets of APN maps is originated  in \cite{czerwinski2020minimal}. Beside the lower bound  in \cite{czerwinski2020minimal}
several  properties and examples of the  image sets of APN maps are presented. In this
paper we develop the study of image sets of APN maps further. Our results indicate 
that the APN maps with minimal image size 
 play a major role for understanding fundamental
properties of APN maps. We believe
 that  a deeper  analysis of the image sets  of APN maps is an interesting research direction which  will allow
to progress in several of current challenges  on APN maps.  
\\

\noindent
Presently, the only known primary families of APN maps are monomials $x\mapsto x^k$ or  Dembowski-Ostrom polynomials. These maps serve as a basis for a handful known secondary constructions of
APN maps \cite{nybergsurvey,pottsurvey}. 
Whereas the image sets of monomial maps are multiplicative subgroups extended with the zero element
and so they are uniquely determined by $\gcd(q-1,k)$,
the behavior of the image sets of Dembowski-Ostrom polynomials is complex and not very well understood yet.
{Results from \cite{carlet2016quadratic} imply that if $n$ is even, then 
 APN Dembowski-Ostrom polynomials of shape $f(x^3)$ 
have the minimal size $(2^n+2)/3$.
For an odd $n$,  the image size of  APN Dembowski-Ostrom polynomials with exponents divisible by $3$ (as integers) is not unique.
We present such families with image sizes $2^n, 2^{n-1}$ and $5\cdot2^{n-3}$ in this paper. }\\

\noindent
At beginning of our studies we were quite certain that having an image set of minimal size
is a  rare property and  not many of known  APN maps will satisfy it.
 Despite our intuition we found that  APN maps constructed by Zhou-Pott and their generalizations suggested
by  G\"olo\u{g}lu are such maps. This is quite remarkable since  these families contain large number
of inequivalent maps as shown in \cite{kaspers}.\\

\noindent
{The set of component-wise plateaued maps includes quadratic maps, and hence also Dembowski-Ostrom polynomials.
For $n$ even,  component-wise plateaued maps with certain preimage distribution 
have a very special  Walsh spectrum. For APN maps this result implies that almost-$3$-to-$1$
component-wise plateaued maps have the classical Walsh spectrum, as observed in \cite{carletplateaued}.
This combined with  the knowledge on the behavior of image sets explains why several important
families of APN maps have 
the classical Walsh spectrum.}
For $n$ odd we find a direct connection between the image set of an almost bent map and the number  of its balanced component functions. As a consequence, we show that any almost bent map has a balanced component function.
We conclude our paper with an upper bound on the image size of non-bijective almost bent maps and component-wise plateaued APN maps. To our knowledge these are the only currently known non-trivial
upper bounds on the image size of APN maps.\\

\section{Images of $d$-uniform functions}\label{general}

In this
section we extend some of the results from  \cite{coulter2011number,kyureghyan2008some,weng2007pseudo} on the image sets of planar maps with $d=1$ to cover a general $d$. 
The behavior of the image sets of $d$-uniform maps and the proofs  are more
complex  for $d>1$. This is simply explained by the fact that the preimage distribution
of a difference map $f_a: x\mapsto f(x+a)-f(x)$ is not unique and more difficult to control
when $d>1$. \\

\noindent
Let $\Image(f)$  be the image set of a map $f:\F_q \to \F_q$. For $r\geq 1$ we denote by  $M_r(f)$  the number of $y\in\F_q$ with exactly $r$ preimages. Further, let $N(f)$ denote the number of pairs $(x,y)\in\F_q^2$, such that $f(x)=f(y)$.
Note $N(f)\geq q$ and
$N(f) =q$ exactly when $f$ is a permutation on $\F_q$.  Let $m$ be the degree of the map $f$, that is the degree
of its polynomial representation of degree not exceeding $q-1$. Then $M_r(f) = 0$ for every $r>m$. 
The following identities follow directly from the definition of $M_r(f)$ and $N(f)$
\begin{align}
	\sum_{r=1}^m M_r(f) &= |\Image(f)| \label{eq_Mr}\\
	\sum_{r=1}^m rM_r(f) &= q \label{eq_rMr}\\
	\sum_{r=1}^m r^2M_r(f) &= N(f) \label{eq_r2Mr}.
\end{align}
The quantities  $M_r$ and $N(f)$ appear naturally when studying the image sets of maps on finite fields, 
see for example \cite{carletdingnonlinearities,coulter2011number,kyureghyan2008some}.  
A map $f:\F_q \to \F_q$ is called $k$-to-1, if every element in the image of $f$ has exactly $k$ preimages, that is
if $M_r(f) =0$ for any $0<r\neq k$.

\begin{lemma}\label{Lem_Cauchy_Schwarz}
	Any map $f:\F_q\to\F_q$ fulfills \[|\Image(f)|\geq\frac{q^2}{N(f)},\]
	with equality  if and only if $f$ is $k$-to-1.
\end{lemma}

\begin{proof}
	It follows from the Cauchy-Schwarz inequality with (\ref{eq_Mr}),~(\ref{eq_rMr}) and (\ref{eq_r2Mr}) that
	\[
		q^2=\left(\sum_{r=1}^m r M_r(f)\right)^2 \leq
		\left(\sum_{r=1}^m r^2M_r(f)\right)\left(\sum_{r=1}^m M_r(f)\right)=N(f)|\Image(f)|.
	\]
	The equality above holds if and only if there is a $k \in \R$ such that $r\sqrt{M_r(f)}=k\sqrt{M_r(f)}$ for all $1\leq r\leq m$, that is $M_k(f)=|\Image(f)|$ and $M_r(f)=0$ for $r\neq k$.
	\qed
\end{proof}

The following proof is an adaption for any $d$ of \cite[Lemma 2]{kyureghyan2008some}, where  planar functions with
$d=1$ were considered.

\begin{lemma}\label{Lem_Nfg}
	Let $f:\F_q\to\F_q$ be $d$-uniform. Then
	\[
		N(f) \leq q+d\cdot t_0(f),
	\]
	where $t_0(f)$ is the number of elements $a\ne 0$ in $\F_q$ for which
	$f(x+a)-f(x)=0$ has a solution $x$ in $\F_q$. The equality holds if and
	only if every of these $t_0(f)$ equations has exactly $d$ solutions. 
\end{lemma}

\begin{proof}
	Note that
	\begin{align*}
		N(f) &= |\{(u, v)\in\F_q^2: f(u) = f(v)\}| \\
				 &= |\{(u, v)\in\F_q^2: f(u)-f(v) = 0\}| \\
				 &= |\{(a, v)\in\F_q^2: f(v+a)-f(v) = 0\}|.
	\end{align*}
	For $a=0$ every pair $(0,v)$ with $v\in\F_q$ contributes to $N(f)$. If $a\neq 0$, then $f(v+a)-f(v)=0$ has at most $d$ solutions because $f$ is $d$-uniform. Therefore
	\[
		N(f)\leq q+ d\cdot t_0(f).
	\]
		\qed
\end{proof}

Observe that for a planar map $N(f) = 2q-1$, since $f(v+a)-f(v)=0$ has a unique solution for every non-zero $a$.
Generalizing this, a map $f:\F_q \to \F_q$ is called zero-difference $d$-balanced if the equation 
$f(x+a)-f(x)=0$ has exactly $d$ solutions for every non-zero $a$, see \cite{carlet2016quadratic}. Hence $N(f) = q + (q-1)d = (d+1)q-d$ for a zero-difference $d$-balanced map.

\begin{corollary}\label{cor:Nf}
	Let $f:\F_q\to\F_q$ be $d$-uniform. Then
	\[
		N(f) \leq (d+1)q-d.
	\]
	The equality holds if and only if $f$ is zero-difference $d$-balanced.
\end{corollary}
\begin{proof}
The statement follows from Lemma \ref{Lem_Nfg} and $t_0(f) \leq q-1$.
	\qed
\end{proof}

\begin{remark}
	Note that several of  the results in this paper hold  for any map $f$ with $N(f)\leq (d+1)q-d$,
	and not only for $d$-uniform ones. Some of our proofs can easily be adapted if $N(f)=kq\pm \varepsilon$ is known.
\end{remark}

\begin{theorem}\label{thm:Image_Size1}
	Let $f:\F_q\to\F_q$ be $d$-uniform. Then
	\[
		|\Image(f)|\geq \left\lceil\frac{q}{d+1}\right\rceil.
	\]
\end{theorem}

\begin{proof}
	With Lemma~\ref{Lem_Cauchy_Schwarz} and Corollary~\ref{cor:Nf} it follows that
	\begin{align}\label{equality}
		|\Image(f)|&\geq \left\lceil\frac{q^2}{N(f)}\right\rceil
			\geq \left\lceil\frac{q^2}{(d+1)q-d}\right\rceil 
			\geq \left\lceil\frac{q}{d+1}\right\rceil.
	\end{align}
		\qed
\end{proof}

The proof of Theorem~\ref{thm:Image_Size1} shows that the gap between  $|\Image(f)|$ and 
 $\left\lceil\frac{q}{d+1}\right\rceil$ is small, 
 when $f$ is close to being $k$-to-1 and  $N(f)$ is about $(d+1)q-d$. Furthermore, 
the bound in  Theorem~\ref{thm:Image_Size1} is sharp; if $d+1$ is a divisor of $q-1$, then
the map $m(x)= x^{d+1}$ reaches the lower bound of Theorem~\ref{thm:Image_Size1} and it is $d$-uniform.
To see that $m(x)$ is indeed $d$-uniform observe that for any non-zero $a$ the difference map 
$m(x+a)-m(x) = (x+a)^{d+1}-x^{d+1}$ has degree $d$  and if $\omega \ne 1$ with $\omega ^{d+1}=1$ then
$x:= (\omega -1)^{-1}$ satisfies $(x+1)^{d+1}-x^{d+1} =0$.
\\

Theorem~\ref{thm:Image_Size2} extends \cite[Theorem 2]{kyureghyan2008some} to cover an arbitrary $d$. Besides of giving a different proof for Theorem~\ref{thm:Image_Size1}, it  additionally provides information on the possible preimage distribution of a $d$-uniform map with minimal image set. 

 For a map $f:\F_q\to\F_q$ and $S\subseteq\F_q$, $a\in\F_q$, we denote by $f^{-1}(S)$ the preimage of $S$ under $f$ and by $\omega(a)$ the size of $f^{-1}(\{a\})$. 

\begin{theorem}\label{thm:Image_Size2}
	Let $f:\F_q\to\F_q$ be $d$-uniform.  Then
	\[
		|\Image(f)|\geq \left\lceil\frac{q}{d+1}\right\rceil.
	\]
	If 
	\[
		|\Image(f)| = \left\lceil\frac{q}{d+1}\right\rceil = \frac{q+\varepsilon}{d+1}
	\]
	with $1\leq\varepsilon\leq d$, then
	\begin{equation}\label{eq_min_image_set1}
		\sum_{y\in \Image(f)}(\omega(y)-(d+1))^2 \leq (d+1)(\varepsilon-1) + 1.
	\end{equation}
\end{theorem}

\begin{proof}
	By Corollary~\ref{cor:Nf}
	\[
	\sum_{y\in \Image(f)}(\omega(y))^2 = \sum_{y\in\F_q} (\omega(y))^2 = N(f) \leq (d+1)q-d.
	\]
	Since \[\sum_{y\in \Image(f)}\omega(y) = q,\]
we get
	\begin{align*}
		0 &\leq \sum_{y\in \Image(f)}(\omega(y)-(d+1))^2 = \sum_{y\in \Image(f)}((\omega(y))^2-2(d+1)\omega(y)+(d+1)^2) \\
			&= N(f)-2(d+1)q+(d+1)^2|\Image(f)| \leq (d+1)q-d-2(d+1)q+(d+1)^2|\Image(f)| \\
			&= -(d+1)q-d+(d+1)^2|\Image(f)|.
	\end{align*}
Hence
	\[
		(d+1)^2|\Image(f)| \geq (d+1)q+d
	\]
	and
	\[
		|\Image(f)|\geq \left\lceil\frac{(d+1)q+d}{(d+1)^2}\right\rceil
				  = \left\lceil\frac{q}{d+1}+\frac{d}{(d+1)^2}\right\rceil
			 \geq \left\lceil\frac{q}{d+1}\right\rceil.
	\]
	Now let \[|\Image(f)| = \left\lceil\frac{q}{d+1}\right\rceil = \frac{q+\varepsilon}{d+1}\] with $1\leq\varepsilon\leq d$. Then
	\begin{align*}
		\sum_{y\in \Image(f)}(\omega(y)-(d+1))^2 &\leq -(d+1)q-d+(d+1)^2\frac{q+\varepsilon}{d+1}\\
			&= \varepsilon d -d + \varepsilon=(d+1)(\varepsilon-1) + 1.
	\end{align*}
		\qed
\end{proof}

Later we use the following observation: 	Let $f$ be a $d$-uniform map and \[|\Image(f)| = \frac{q+\varepsilon}{d+1}.\] Define
	\[
		D = \{y\in \Image(f): \omega(y)\neq d+1\}.
	\]
	Then we have
	\[
		q = \sum_{y\in \Image(f)} \omega(y)
			= \sum_{y\in \Image(f)\setminus D} \omega(y)  + \sum_{y\in D} \omega(y)
			= \left(\frac{q+\varepsilon}{d+1}-|D|\right)(d+1) + \sum_{y\in D} \omega(y),
	\]
	implying
	\begin{equation}\label{eq_min_image_set2}
		\sum_{y\in D}\omega(y) = |D|(d+1)-\varepsilon.
	\end{equation}

The following theorem is a generalization of \cite[Theorem 1]{coulter2011number} and it provides
information on the possible preimage distribution of a $d$-uniform map.

\begin{theorem}\label{th:2.7}
	Let $f:\F_q \to \F_q$ be $d$-uniform. Then
	\begin{equation} \label{th-1case}
		\sum_{r=1}^d r(d+1-r)M_r(f)\geq d
	\end{equation}
	and
	\begin{equation} \label{th-2case}
		\sum_{r=1}^{d+1}r(d+2-r)M_r(f) \geq q+d.
	\end{equation}
	{
	The equality in \eqref{th-1case} holds if and only if $N(f)=(d+1)q-d$ and $M_r(f)=0$ for all $r \geq d+2$; and the equality in  \eqref{th-2case} holds 
	 if and only if $N(f)=(d+1)q-d$ and $M_r(f)=0$ for $r>d+2$. The latter case reduces to
	\[
		\sum_{r=1}^d r(d+1-r)M_r(f) = 
		(d+2)M_{d+2}(f)+d.
	\]
	}
\end{theorem}

\begin{proof}
	Let $m$ be the degree of $f$. With Corollary~\ref{cor:Nf} we have $N(f)\leq(d+1)q-d$. Using \eqref{eq_rMr} and \eqref{eq_r2Mr} we get
	\[
		\sum_{r=1}^m r^2M_r(f) = N(f) \leq (d+1)q-d = (d+1)\sum_{r=1}^m rM_r(f) - d,
	\]
	so that
	\begin{equation}\label{eq1}
		\sum_{r=1}^{d+1} (r(d+1)-r^2)M_r(f) -d\geq
		\sum_{r=d+2}^m (r^2-(d+1)r)M_r(f)
	\end{equation}
	As the right hand side is non-negative, we have in particular
	\[
		\sum_{r=1}^d r(d+1-r)M_r(f)\geq d,
	\]
		with equality if and only if $M_r(f)=0$ for all $r \geq d+2$ and $N(f) = (d+1)q-d$.
	Note that for $r\geq d+2$ it holds that $r^2-(d+1)r\geq r$, so that \eqref{eq1} turns into
	\begin{equation}\label{eq2}
		\sum_{r=1}^{d+1} r(d+1-r)M_r(f)\geq
		\sum_{r=d+2}^m (r^2-(d+1)r)M_r(f)+d\geq
		\sum_{r=d+2}^m rM_r(f) + d.
	\end{equation}
	Adding $\sum_{r=1}^{d+1}rM_r(f)$ on both sides of \eqref{eq2} and using \eqref{eq_rMr} gives
	\[
		\sum_{r=1}^{d+1} r(d+2-r)M_r(f) \geq
		\sum_{r=1}^m rM_r(f) +d = q+d.
	\]
	For equality to hold, we need equality in \eqref{eq2}. The first equality in \eqref{eq2} holds if and only if $N(f)=(d+1)q-d$, the second equality holds if and only if
	\[
		\sum_{r=d+2}^m (r^2-(d+1)r)M_r(f) = 
		\sum_{r=d+2}^m rM_r(f),
	\]
	that is $M_r(f)=0$ for $r> d+2$. In that case
	\[
		\sum_{r=1}^d r(d+1-r)M_r(f) = 
		(d+2)M_{d+2}(f)+d.
	\]
		\qed
\end{proof}

\section{Dembowski-Ostrom $d$-uniform polynomials}

A polynomial/map $f \in \F_q[x]$ is called Dembowski-Ostrom (DO), if it can be written as
\[
	f(x) = \sum_{i,j=0}^{n-1}a_{ij}x^{p^i+p^j}
\]
 when $q$ is odd and
 \[
	 f(x) = \sum_{\substack{i,j=0\\i\neq j}}^{n-1}a_{ij}x^{2^i+2^j}
 \]
 when $q$ even. Note that $x^2$ is a DO polynomial for any odd $q$, but not for even $q$.
Maps obtained as the sum of a DO map with an $\F_p$-affine one are called quadratic.
 
Let $k$ be a divisor of $q-1$. We call a map $f$  $k$-divisible, if it can be written as $f(x)=f'(x^k)$ for a suitable $f'$.  
Observe that $f$ is $k$-divisible if and only if  $f(x)=f(\omega x)$ for all $x\in \F_q$ and all $\omega \in \F_q^*$ whose order divides $k$. Further, we call a map $f$ almost-$k$-to-1\footnote{Note that in many papers such maps are called just
$k$-to-1. However we use the terminology almost-$k$-to-1 to avoid confusion with $k$-to-1 maps considered in Section \ref{general}.}, if  there is a unique element in  $\Image(f)$ with exactly 1 preimage and all other images have exactly $k$ preimages. 
Note that if $f(x)$ is $d$-uniform, then so is $f(x+c)+u$ for arbitrary $c,u\in\F_q$. 
Hence we may without loss of generality assume that $f(0)=0$ and that 0 is the unique element with exactly one preimage,
when considering the $d$-uniform property of an almost-$k$-to-$1$ map $f$.\\

For a non-zero $a\in \F_q$ we  define 
$$D_a(f):= \{f(x+a)-f(x) : x\in \F_q\},$$ 
which we call a differential set of $f$ in direction $a$. It is well-known and easy to see
that the differential sets of quadratic maps are $\F_p$-affine subspaces. The following result can be partly
deduced from Proposition 3 and Corollary 1 and their proofs in \cite{carlet2016quadratic}. We include its proof for the convenience of
the reader.

\begin{lemma}\label{Lem_hyperplane}
Let $q=p^n$ with $p$ prime, $d+1$ be a divisor of $q-1$ and $f:\F_q \to \F_q$ be a $(d+1)$-divisible DO polynomial which is almost-$(d+1)$-to-1. Then 
\begin{itemize}
\item[(a)]
$f$ is zero-difference $d$-balanced;
\item[(b)]
 $f$ is $d$-uniform and all its differential sets are $\F_p$-linear subspaces;
\item[(c)] $d=p^i$ for some $i\geq 0$.
\end{itemize}
\end{lemma}
\begin{proof} First we prove statements (a) and (b):
Since $f$ is a DO polynomial, it is $d$-uniform in the case  it is zero-difference $d$-balanced.
First we show that for any non-zero $a$ the equation $f_a(x) = f(x+a)-f(x) =0$ has a solution
(equivalently, $D_a(f)$ is a subspace).
Indeed, let $1\neq\omega\in\F_q$ be a zero of $x^{d+1}-1$ and set $x=(\omega-1)^{-1}a$.
This $x$ fulfills $x+a = \omega x$, and hence $f_a(x) = f(\omega x)-f(x) =0$.
In particular, $f_a(x)=0$ hast at least $d$ solutions.
On the other side, since $f$ is $(d+1)$-divisible and almost-$(d+1)$-to-1, the equation $f(x+a)=f(x)$ 
is fulfilled if and only if $x+a = \omega x$ for an element $\omega$ satisfying $\omega^{d+1}=1$.
This implies that a solution $x$ must be given by $a(\omega-1)^{-1}$. And hence there are
at most $d$ solution for $f_a(x)=0$. 

The statement in (c) follows from (b). Indeed,
the differential sets of $f$ are linear subspaces of size $p^n/d$, and hence  $d=p^i$ for some $i \geq 0$.
	\qed
\end{proof}

The following result holds if $f$ is not a DO polynomial too:

 \begin{theorem}\label{Thm_d+1to1_if_d-uniform}
	 Let $d+1$ be a divisor of $q-1$ and $f:\F_q\to\F_q$ be $(d+1)$-divisible and $d$-uniform. Then $f$ is almost-$(d+1)$-to-1.
 \end{theorem}
 
 \begin{proof}
	 As $f$ is $(d+1)$-divisible, we have $|\Image(f)|\leq \frac{q+d}{d+1}$. By Theorem~\ref{thm:Image_Size1} we have $|\Image(f)|\geq \frac{q+d}{d+1}$ and therefore $|\Image(f)|= \frac{q+d}{d+1}$ and $f$ is almost-$(d+1)$-to-1.
		\qed
 \end{proof}

A fascinating property of DO planar polynomials proved in \cite{coulter2011number,weng2012further} is:
A DO polynomial is planar if and only if it is  almost-2-to-1. Observe that for an odd $q$ a DO polynomial
is always $2$-divisible. Corollary 1  in \cite{carlet2016quadratic} proves an  analog of this result for the $d$-uniform case;
Theorem \ref{th_divisible} is a reformulation of it using the terminology
introduced in this paper.

\begin{theorem}\label{th_divisible}
	Let $d+1$ be a divisor of $q-1$. A $(d+1)$-divisible DO polynomial $f$ is 
	$d$-uniform if and only if $f$ is almost-$(d+1)$-to-1.
\end{theorem} 
\begin{proof}
It follows directly from Theorem~\ref{Thm_d+1to1_if_d-uniform} and  Lemma~\ref{Lem_hyperplane}. \qed
\end{proof}

Note that if $d=p^i$ and $d+1=p^i+1$ is a divisor of $q-1$, the map $x\mapsto x^{d+1}$ is a $d$-uniform DO Polynomial that is almost-$(d+1)$-to-1.
\\

\section{Image sets of APN maps of binary finite fields}

In the following sections we study the image sets of APN maps on binary fields. Such maps are of particular interest
because of their  applications in cryptography and combinatorics.  \\

\noindent
 {The first systematic
study of the image sets of APN maps was originated  in \cite{czerwinski2020minimal}, where it is shown 
  that the image set of an APN map on $\F_{2^n}$ contains at least $\lceil (2^n+1)/3 \rceil$ elements. This lower bound
	is proved by methods of linear programming in \cite{czerwinski2020minimal}, which is a novel approach for studying image sets of maps on finite fields. 
	The APN monomials have the image size  $(2^n+2)/3$ for $n$ even, showing that the lower bound
	is sharp for $n$ even. Several numerical results presented in \cite{czerwinski2020minimal} 
	 suggest that  the minimal image size of APN maps is much larger,
	probably around $2^{n-1}$, if $n$ is odd.}
	{Preprint
	\cite{carlet2020} comments that the lower bound   on the image sets of APN maps appears (in an equivalent form) already in \cite[Lemma 5]{carlet-heuser-picek}, where a lower bound on the differential
	uniformity via image set size is presented. The  arguments proving  Lemma 5 in \cite{carlet-heuser-picek} 
	are similar to
	ours  presented for Lemma \ref{Lem_Cauchy_Schwarz} and Corollary \ref{Lem_Nfg}. 
	These are more or less standard for studying image sets of 
	maps with special additive properties on finite sets, see \cite{kyureghyan2008some,weng2007pseudo} .}
	\\
	
	{Results of Section \ref{general}   allow, beside  the lower bound on the image size, to describe also the possible preimage distributions of APN maps meeting  it, see Theorem \ref{thm:APN_Lower_Bound}. For the APN maps on $\F_{2^n}$ Theorem~\ref{th:2.7}  reduces to:}
	
	\begin{corollary}\label{cor:apn-1-2} 
	Let $f:\F_{2^n} \to \F_{2^n}$ be APN. Then
	\begin{itemize}
	\item[(a)]
	\[
		M_1(f)+M_2(f)\geq 1,
	\]
	and hence there is at least one element with exactly 1 or 2 preimages. 
	 For $n$ even, the inequality is sharp if and only if $f$ is almost-$3$-to-$1$.
	For $n$ odd, the inequality is sharp if and only if  there is a unique element in $\Image(f)$
	with exactly two preimages and the remaining elements have exactly three preimages.
  \item[(b)] 
	\[
		3M_1(f) + 4M_2(f) + 3M_3(f) \geq 2^n+2.
	\]
	The equality holds if and only if $N(f)=3\cdot 2^n-2$ and $M_r(f)=0$ for $r>4$, in which case
	\[
		M_1(f)+M_2(f) = 2M_4(f)+1.
	\]
	\end{itemize} 
\end{corollary} 
\begin{proof} {The inequalities as well as the equality case in (b) follow directly from Theorem~\ref{th:2.7}. 
Let $M_1(f)+M_2(f) = 1$. Then $M_r(f) = 0$ for every $r\geq 4$ by Theorem~\ref{th:2.7}.  To complete the proof note that the value  $2^n \pmod 3$ forces $(M_1(f),M_2(f)) =(1,0)$ resp. $(M_1(f),M_2(f)) =(0,1)$
depending on the parity of $n$.  \qed}
\end{proof} 

{The observation that an APN map must have at least one element with exactly 1 or 2 preimages was done
already in \cite{czerwinski2020minimal}.}\\

\noindent	
	The next theorem is a consequence of Theorem~\ref{thm:Image_Size1} and Theorem~\ref{thm:Image_Size2}. 
	Corollary~\ref{cor:apn-1-2} along with 
identity \eqref{eq_min_image_set2}  and inequality \eqref{eq_min_image_set1} yield the possible preimage distributions of  an APN map meeting  the lower bound.

\begin{theorem}\label{thm:APN_Lower_Bound}
	Let $f:\F_{2^n} \to \F_{2^n} $ be APN. Then
	\[
		|\Image(f)|\geq 
		\begin{cases}
			\frac{2^n+1}{3} & n \text{ is odd,} \\
			\frac{2^n+2}{3} & n \text{ is even}.
		\end{cases}
	\]
	If $n$ is odd and \[|\Image(f)| = \frac{2^n+1}{3},\] then $\omega(y_0)=2$ for one element $y_0\in \Image(f)$ and $\omega(y)=3$ for $y\in \Image(f)\setminus\{y_0\}$.
	
	If $n$ is even and \[|\Image(f)| = \frac{2^n+2}{3},\] then  one of the following cases must occur:
	\begin{enumerate}
		\item $\omega(y_0)=1$ for one element $y_0\in \Image(f)$ and $\omega(y)=3$ for all 
		$y\in \Image(f)\setminus\{y_0\}$,
		that is $f$ is almost-3-to-1.
		
		\item $\omega(y_i)=2$ for two elements $y_0, y_1\in \Image(f)$ and $\omega(y)=3$ for all 
		$y\in \Image(f)\setminus\{y_0, y_1\}$. 
		
		\item $\omega(y_i)=2$ for three elements $y_0, y_1, y_2\in \Image(f)$, $\omega(y_3)=4$ for a unique
		$y_3\in \Image(f)\setminus\{y_0, y_1, y_2\}$ and $\omega(y)=3$ for all  
		$y\in \Image(f)\setminus\{y_0, \ldots, y_3\}$. 
	\end{enumerate}
\end{theorem}
{
\begin{proof}
	Theorem~\ref{thm:Image_Size2} immediately yields the lower bound.
	We apply now \eqref{eq_min_image_set1} and \eqref{eq_min_image_set2} to prove
	the statements on the preimage distribution. Set $D = \{y\in \Image(f): \omega(y)\neq 3\}$.  If $n$ is odd, 
	by \eqref{eq_min_image_set1} we get
		\[\sum_{y\in D}(\omega(y)-3)^2 \leq 1.\]
		Hence there is at most one $y_0 \in \Image(f)$ such that $\omega(y)\neq 3$ and  it must satisfy
		$\omega(y_0) \in\{2,4\}$.  Corollary \ref{cor:apn-1-2} completes the proof.
		Let  $n$ be even. Then from \eqref{eq_min_image_set1} and \eqref{eq_min_image_set2} we get
		\[\sum_{y\in D}(\omega(y)-3)^2 \leq 4\]
	and
		\[\sum_{y\in D}\omega(y) = 3|D|-2.\]
		Clearly, if $|D|=1$, then $f$ is almost-$3$-to-$1$. If $|D|=2$, then $\omega(y)=2$ for every $y \in D$. Note that $|D|=3$ is not possible, since  $\omega(y)\in \{2,4\}$ for all $y \in D$, contradicting the second equation since $3|D|-2$ is odd in this case. If $|D|=4$, we have again $\omega(y)\in \{2,4\}$ for all $y \in D$ and the only solution to the second equation is $\omega(y)=2$ for $3$ elements and $\omega(y)=4$ for one element. $|D|>4$ violates the first equation, so we exhausted all possibilities.
	\qed
\end{proof}}

The  APN monomials meet the lower bound for $n$ even; we present
several further such families  later in this paper. All these examples of  APN maps are almost-3-to-1.
We believe that cases 2. and 3. for $n$ even  never occur. \\

\noindent
{ {\bf Open Problem:} Let $n$ be even and $f:\F_{2^n} \to \F_{2^n} $ be APN map with
$|\Image(f)| = (2^n+2)/3$. Can $f$ have the preimage distribution described in case 2. or 3.
of Theorem \ref{thm:APN_Lower_Bound}\,?}\\

\noindent
 Numerical results  suggest that there are no  APN maps meeting the lower bound for $n$ odd. 
We show later in this paper that the image sizes of almost bent maps never fulfill the 
lower bound in Theorem~\ref{thm:APN_Lower_Bound}.
The APN maps with smallest sizes which we found are 
\begin{itemize}
\item[-] for $n=7$ the map $x \mapsto x^3+x^{64}+x^{16}+x^4$ with the image size $57=2^6-7$; 
\item[-] for $n=11$ the  map $x \mapsto x^3 +x^{256} $  with the image size $1013=2^{10}-11$. 
\end{itemize} 
In \cite{czerwinski2020minimal} it is shown that the APN binomial $b(x)= x^3+x^4$ 
is $2$-to-1 if $n$ is odd. This binomial is studied in \cite{kyureg-mueller-wang}:
For an even $n$  the image set of  $b(x)=x^3+x^4$ satisfies
$M_1(b) = 2(2^n-1)/3, ~ M_2(b)=1$ and $M_4(b)=(2^n-4)/12$, and hence  $|\Image(b)|= 3\cdot 2^{n-2}$.\\

\noindent
The lower bound in Theorem~\ref{thm:APN_Lower_Bound} can be used to prove several structural
results for APN maps. For example, it  gives an easy proof for the following well-known property of
monomial APN maps.

\begin{corollary}
	Let $q=2^n$ and $f(x)=x^k$ be APN on $\F_{q}$. Then $\gcd(k, q-1)=1$ if $n$ is odd and $\gcd(k, q-1)=3$ if $n$ is even.
\end{corollary}

\begin{proof}
	Since \[|\Image(f)\setminus\{0\}| = \frac{q-1}{\gcd(k, q-1)},\]
	 Theorem~\ref{thm:APN_Lower_Bound} forces $\gcd(k, q-1)\leq 3$. For $n$ odd we get $\gcd(k, q-1)=1$. 
	Now let $n$ be even and $\gcd(k, q-1)=1$. Then $f$ is an APN permutation on all subfields of $\F_q$. In particular it must be an APN permutation on $\F_4$. It is easy to check that such a permutation does not exist. Hence $\gcd(k, q-1)=3$.
		\qed
\end{proof}

  The following
characterization for APN $3$-divisible DO polynomials is a direct consequence of Theorem~\ref{th_divisible}:

 \begin{corollary}\label{th_iff}
	 Let  $n$ be even and $f:\F_{2^n}\to \F_{2^n}$ be a 3-divisible DO polynomial. Then $f$ is APN if and 
	only if $f$ is almost-3-to-1.
 \end{corollary}

We take a closer look at 3-divisible APN maps in the next section.\\

Next we observe that   Zhou-Pott APN quadratic maps constructed 
in \cite{zhou2013new} provide examples of
 almost-3-to-1 APN maps   which are not 3-divisible. 

\begin{theorem} \label{thm:zhoupott31}
	Let  $m, i\geq 2$ even, $\gcd(k, m)=1$ and $\alpha\in\F_{2^m}$ not a cube. 
	\begin{itemize}
	\item[(a)]
	Then $f:\F_{2^m}\times\F_{2^m}\to\F_{2^m}\times\F_{2^m}$ given by
\begin{equation}
	f(x,y) = (x^{2^k+1}+\alpha y^{(2^k+1)2^i}, xy)
\label{eq:zhoupott}
\end{equation}
 is almost-3-to-1. More precisely $f(x,y) = f(u,v)$ if and only if $(x,y)=(\omega u, \omega^2 v)$ 
with $\omega \in \F_4^*$. { The corresponding to $f(x,y)$ 
map on $\F_{2^{2m}}$ has a univariate representation which is not a DO-polynomial.}
\item[(b)]
	Then $g:\F_{2^m}\times\F_{2^m}\to\F_{2^m}\times\F_{2^m}$ given by
\begin{equation}
	g(x,y) = (x^{2^k+1}+\alpha y^{(2^k+1)}, xy^{2^{m-i}})
\label{eq:zhoupott}
\end{equation}
 is almost-3-to-1. {The corresponding to $g(x,y)$ 
map on $\F_{2^{2m}}$ has a univariate representation which is a DO-polynomial that is not 3-divisible.}
\end{itemize}
\end{theorem}
\begin{proof}
Note that $\gcd(2^{2m}-1, 2^k+1) =3$ and 3 is a divisor of both $2^m-1$ and $2^i-1$.\\

(a)	Let $(x,y),(u,v)\in\F_{2^m}\times\F_{2^m}$ with $f(x,y)=f(u,v)$. Then we have
	\begin{align*}
		x^{2^k+1}+\alpha y^{(2^k+1)2^i} &= u^{2^k+1}+\alpha v^{(2^k+1)2^i} \\
		xy &= uv.
	\end{align*}
	First suppose $v=0$, and hence $x=0$ or $y=0$ too. For $x=0$, we get
	$$
	\alpha y^{(2^k+1)2^i} = u^{2^k+1},
	$$
	which forces $y=u=0$, since $\alpha$ is a non-cube. For $x, u \ne 0$ and $y=0$ we get
	$$
	x^{2^k+1} = u^{2^k+1},
	$$
	which is satisfied if and only if $x = \omega u$ with $\omega   \in \F_4^*$.
	Now let $v\ne 0$. Setting $u=\frac{xy}{v}$ and rearranging the first equation we get
	\[
		x^{2^k+1}+\left(\frac{xy}{v}\right)^{2^k+1} = \alpha (y^{2^k+1}+v^{2^k+1})^{2^i}
	\]
	or equivalently
	\[
		x^{2^k+1}\left(1+\left(\frac{y}{v}\right)^{2^k+1}\right) = \alpha v^{(2^k+1)2^i}\left(1+\left(\frac{y}{v}\right)^{2^k+1}\right)^{2^i}.
	\]
	If $1+\left(\frac{y}{v}\right)^{2^k+1}\neq 0$, we can divide by it and obtain
	\begin{equation}
	 x^{2^k+1} = \alpha v^{(2^k+1)2^i}\left(1+\left(\frac{y}{v}\right)^{2^k+1}\right)^{2^i-1}.
	\label{eq:cube_eq}
	\end{equation}
	Note that \eqref{eq:cube_eq} has no solution, since $x^{2^k+1}$, $ v^{(2^k+1)2^i}$ and $\left(1+\left(\frac{y}{v}\right)^{2^k+1}\right)^{2^i-1}$ are all cubes and $\alpha$ is not a cube.
	Finally observe that $\left(\frac{y}{v}\right)^{2^k+1}= 1$ holds if and only if  $y =\omega v$ with $\omega   \in \F_4^*$.\\
	Next we show that the corresponding to $f(x,y)$ map on $\F_{2^{2m}}$ is not given by a univariate DO 
	polynomial.
Let $(u_1, u_2)$ be a basis of $\F_{2^{2m}}$ over $\F_{2^m}$ and $(v_1, v_2)$ its dual basis.
Then an element $z$ of $\F_{2^{2m}}$ has the  representation $(v_1z+\overline{v_1z})u_1 + (v_2z+\overline{v_2z})u_2$,
where $\overline{a} = a^{2^m}$. Thus we get 
\begin{eqnarray*}
f(z) & = &f(v_1z+\overline{v_1z}, v_2z+\overline{v_2z}) \\
 &= &\left((v_1z+\overline{v_1z})^{2^k+1}+\alpha (v_2z+\overline{v_2z})^{(2^k+1)2^i}\right)u_1 \\ &+& (v_1z+\overline{v_1z})\cdot (v_2z+\overline{v_2z}) u_2 \\
 & = & \ldots +( (v_1\overline{v_2} + \overline{v_1}v_2)z^{2^m+1}+ v_1v_2z^2 + \overline{v_1v_2}z^{2\cdot2^m})u_2.
\end{eqnarray*}
Since $k\ne 0$, there will be no term $z^2$ in the summand for $u_1$, and hence the above polynomial
contains a non-zero term with $z^2$, showing that $f(z)$ is not a DO-polynomial.\\

(b) Note that $g(x,y)$ is obtained from $f(x,y)$ by a linear bijective transformation
$(x,y) \mapsto (x, y^{2^{m-i}})$. In particular, $g(x,y)$ is almost 3-to-1, too.
Next we describe the univariate representation of the corresponding to $g(x,y)$ map on $\F_{2^{2m}}$.
Again let $(u_1, u_2)$ be a basis of $\F_{2^{2m}}$ over $\F_{2^m}$ and $(v_1, v_2)$ its dual basis.
 Then we get 
\begin{eqnarray*}
g(z) & = &g(v_1z+\overline{v_1z}, v_2z+\overline{v_2z}) \\
 &= &\left((v_1z+\overline{v_1z})^{2^k+1}+\alpha (v_2z+\overline{v_2z})^{(2^k+1)}\right)u_1 \\ &+& (v_1z+\overline{v_1z})\cdot (v_2z+\overline{v_2z})^{2^{m-i}} u_2,
\end{eqnarray*}
which is a DO polynomial. Finally, note that $g(x,y) \ne g(\omega x, \omega y)$ for $\omega \in \F_4 \setminus \F_2$, and hence the DO polynomial $g(z)$ is not 3-divisible.

	\qed
\end{proof}

\begin{remark} {
Observe that our proof of Theorem \ref{thm:zhoupott31} does not use the property that $f(x,y)$ is APN.
Hence Theorems \ref{thm:zhoupott31} and \ref{thm:walsh_spectrum_three_to_one} prove that the maps $f(x,y)$ and $g(x,y)$ are APN. }
\end{remark}

{In ~\cite{carlet2016quadratic} a map $f:\F_q \to \F_q$ is called $\delta$-vanishing  
if for any non-zero $a$ the equation $f(x+a)-f(x) =0$ has $t_a$ solutions, where  $ 0 < t_a \leq \delta$.
Note that any  zero-difference $d$-balanced map is $d$-vanishing.
Problem 1 in ~\cite{carlet2016quadratic} asks whether any quadratic $\delta$-vanishing map
must be $d$-divisible with an appropriate $d$. Theorem \ref{thm:zhoupott31} shows that the answer
to this problem is negative. Indeed, the Zhou-Pott maps $f(x,y)$ and $g(x,y)$ are quadratic APN maps
which are not $3$-divisible. These maps are almost-3-to-1 and hence $N(f)=N(g)=3q-2$ and then by
Corollary \ref{cor:Nf} they are zero-difference 2-balanced.  } \\

\noindent
{
The next  sufficient condition for an APN map to be almost-3-to-1 
follows immediately from the lower bound  in
Theorem \ref{thm:APN_Lower_Bound}.
\begin{proposition}\label{prop:apn_3to1}
	Let $n$ be even and  $f:\F_{2^n}\to \F_{2^n}$ be an APN map satisfying
	\begin{itemize}
		\item $f(0)=0$,
		\item every $y \in \Image(f)\setminus \{0\}$ has at least $3$ preimages. 
	\end{itemize}
	Then $f$ is almost-3-to-1.
\end{proposition} \qed
The Zhou-Pott construction of APN maps  was recently generalized in ~\cite{farukapn}.
Next we use Proposition \ref{prop:apn_3to1} to show  that the APN maps of this construction  are  almost-$3$-to-$1$ too.
\begin{theorem}\label{apn-faruk}
	Define the following maps on $\F_{2^m} \times \F_{2^m}$
	\[f_1(x,y)=(x^{2^k+1}+xy^{2^k}+y^{2^k+1},x^{2^{2k}+1}+x^{2^{2k}}y+y^{2^{2k}+1})\]
	where $\gcd(3k,m)=1$ and 
	\[f_2(x,y)=(x^{2^k+1}+xy^{2^k}+y^{2^k+1},x^{2^{3k}}y+xy^{2^{3k}}),\]
	where $\gcd(3k,m)=1$ and $m$ is odd.	
	Then $f_1$ and $f_2$ are almost-$3$-to-$1$ APN maps.
\end{theorem}
\begin{proof}
	The APN property of these maps (under the given conditions) was proven in \cite{farukapn}. 
	For the rest, we check the conditions of Proposition~\ref{prop:apn_3to1}. The first condition is clearly satisfied in both cases.
	We start with $f_1$:
	Direct computations show that 
	\begin{align*}
		f_1(y,x+y)&=  \\  & (y^{2^k+1}+y(x+y)^{2^k}+(x+y)^{2^k+1},y^{2^{2k}+1}+y^{2^{2k}}(x+y)+(x+y)^{2^{2k}+1}) \\
		&=(x^{2^k+1}+xy^{2^k}+y^{2^k+1},x^{2^{2k}+1}+x^{2^{2k}}y+y^{2^{2k}+1})=f_1(x,y).
	\end{align*}
	A similar calculation yields $f_1(x,y)=f_1(x+y,x)$. Thus every $y \in \Image(f_1)\setminus \{0\}$ has at least three preimages. Both conditions of Proposition~\ref{prop:apn_3to1} are satisfied for $f_1$, completing the proof for $f_1$.
	Consider now $f_2$:
	Similarly to the first case, we have
	\begin{align*}
		f_2(y,x+y)&=(y^{2^k+1}+y(x+y)^{2^k}+(x+y)^{2^k+1},y^{2^{3k}}(x+y)+y(x+y)^{2^{3k}}) \\
		&=(x^{2^k+1}+xy^{2^k}+y^{2^k+1},x^{2^{3k}}y+xy^{2^{3k}})=f_2(x,y),
	\end{align*}	
	and similarly  $f_2(x,y)=f_2(x+y,x)$, so both conditions of Proposition~\ref{prop:apn_3to1} are again satisfied.
	\qed
\end{proof}
A further large family of inequivalent almost-$3$-to-$1$ APN maps has been found by Faruk G\"olo\u{g}lu and the first author, and will be published soon.} \\

\noindent
A natural question is whether  every quadratic APN map of $\F_{2^n}$ with  even $n$ is EA-equivalent to 
an almost-$3$-to-$1$ map. The answer is negative. By  Theorem~\ref{thm:walsh_spectrum_three_to_one}, all
almost-3-to-1 quadratic APN maps have the classical Walsh spectrum. And hence the EA-class of  quadratic APN maps with non-classical Walsh spectra do not contain  an almost-$3$-to-$1$ map.

\section{3-divisible APN maps}

Observe that by Corollary~\ref{th_iff}, every APN  DO polynomial $f'(x^3)$ on $\F_{2^n}$, $n$ even, is an example
with the preimage distribution described in Case 1. of Theorem~\ref{thm:APN_Lower_Bound}. Prominent examples for such
APN maps are $x\mapsto x^3$ and $x\mapsto x^3 +\Tr(x^9)$.
 These maps are APN for any $n$. If $n$ is odd,
then $x\mapsto x^3$ is a permutation and $x\mapsto x^3 +\Tr(x^9)$ is 2-to-1, as we will see later in this section.

Next we present an interesting observation which could be helpful for performing numerical searches as well as theoretical studies of
3-divisible APN DO polynomials. In particular it could be used for
 classifying exceptional APN 3-divisible DO polynomials.

\begin{theorem}\label{th_subfield_permutation}
	 Let  $n=2^im$ with $i\geq 1$ and $m\geq 3$ odd. 
	Suppose  $f:\F_{2^n}\to \F_{2^n}$ is a 3-divisible APN DO polynomial over the subfield $\F_{2^m}$ that is $f \in \F_{2^m}[x]$. Then $f$ is an APN permutation on the subfield $\F_{2^m}$.
 \end{theorem}

\begin{proof}
Since the coefficients of $f$ are from $\F_{2^m}$, it defines an APN map on it.
By Corollary~\ref{th_iff}, $f$ is almost-3-to-1 on $\F_{2^n}$. Moreover, $f(x)=f(\omega x)=f(\omega^2 x)$ for $\omega \in \F_4\setminus\F_2$.
The statement now follows from the fact that $\F_4$ is not contained in $\F_{2^m}$.
	\qed
\end{proof}

The substitution of $x^3$ in a polynomial of shape $f'(x) = L_1(x) + L_2(x^3)$, 
where $L_1, L_2$ are linearized polynomials, results in
a DO polynomial $f(x) = f'(x^3) = L_1(x^3) + L_2(x^9)$. Hence for even $n$ by Corollary~\ref{th_iff} 
such a map is APN if and only if it is almost 3-to-1. In particular,
any permutation of shape $L_1(x) + L_2(x^3)$  yields directly
an APN DO polynomial if $n$ is even.  Observe that $x^3$ and $ x^3+ \Tr(x^9)$ are
of this type too. These and further   APN DO polynomials $L_1(x^3) + L_2(x^9)$ are studied in \cite{budaghyan2009constructing,budaghyan2009construction}. Corollary~\ref{th_iff} suggests a unified approach
for understanding such APN maps.

Results  from \cite{charpin2008class} can be used to construct and explain  permutations of shape $f'(x) = L_1(x) + L_2(x^3)$.
For example, Theorem 6 in \cite{charpin2008class}  with $s=3$ and $L(x) = x^2+\alpha x$ yields the following family of APN 3-divisible DO polynomials.

\begin{theorem}\label{th:apn_chk}
Let $\alpha, \beta, \gamma$ be non-zero elements in $\F_{2^n}$ with $n$ even. Further let
$\gamma \not\in \{ x^2+\alpha x ~|~ x \in \F_{2^n}\}$ and $\Tr(\beta \alpha)=1$, then 
 $$
f'(x) = x^2 + \alpha x +\gamma \Tr(\alpha^{-3}x^3+\beta x)
$$
is a permutation on  $\F_{2^n}$ and consequently
$$
f(x) = f'(x^3) = x^6 + \alpha x^3 +\gamma \Tr(\alpha^{-3}x^9+\beta x^3)
$$
is APN.
\end{theorem} \qed

An APN map constructed in  Theorem~\ref{th:apn_chk} is affine equivalent to one of form 
$x^3 + \alpha\Tr(\alpha^{-3}x^9)$ studied in \cite{budaghyan2009construction}.
Indeed, the map $f'(x)$ can be reduced to 
$$
f'(x) = x^2 + \alpha x + \gamma \Tr(\beta x) + \gamma \Tr(\alpha^{-3}x^3) = L_1(x)+\gamma \Tr(\alpha^{-3}x^3),
$$
where $L_1$ is linear over $\F_2$. Using \cite[Theorem 5]{charpin2008class} the map $L_1$
is bijective. Then $L_1^{-1}$ composed with $f'(x)$ yields
$$
L_1^{-1} \circ f'(x) =  x + \Tr(\alpha^{-3}x^3)L_1^{-1}(\gamma),
$$
and thus
$$
L_1^{-1} \circ f'(x^3)  =  x^3 + \Tr(\alpha^{-3}x^9)L_1^{-1}(\gamma) = x^3+\alpha \Tr(\alpha^{-3}x^9),
$$
where for the last equality we used $L_1(\alpha) = \gamma$. Note that this reduction remains true for $n$ odd too,
showing that the examples of Theorem~\ref{th:apn_chk} are APN  for any $n$.\\

As we mentioned earlier the polynomials $x^3$ and $x^3+\Tr(x^9)$ define APN maps on $\F_{2^n}$ for every $n\geq 1$.
For $n$ odd, the first map is a permutation and the second is 2-to-1, as the following result shows. 

\begin{proposition}
Let $n$ be odd and $a \in \F_{2^n}$ non-zero.
Then the  APN map $x \mapsto x^3 + a^{-1}\Tr(a^3x^9)$ is 2-to-1.
\end{proposition}

\begin{proof}
This follows from Theorem 3 in \cite{charpin-kyureg-ffa}, since $a^{-1}$ is a $1$-linear structure of $\Tr(a^3x^3)$ for $n$ odd,
which can be easily checked by direct calculations.
	\qed
\end{proof}

We believe that if $n$ is odd, then $2^{n-1}$ is the minimal possible image size of an APN DO polynomial of shape 
$L_1(x^3) +L_2(x^9)$.
 However, an analog of Corollary~\ref{th_iff} is not true for 
the size $2^{n-1}$.
There are such DO polynomials with image size $2^{n-1}$, which are not APN.
For example, if $n$ is odd,  the DO polynomial $(x^2 + x)\circ x^3$ is 2-to-1, but not APN.\\

\noindent

Next we observe that for $n$ odd there are APN DO polynomials of shape $L_1(x^3)+L_2(x^9)$, that are neither bijective nor have image size $2^{n-1}$. 
For a divisor $t$ of $n$ we denote by {$\Tr_{2^n/2^t}(x)$} the trace map from $\F_{2^n}$ into the subfield 
$\F_{2^t}$, that is
{\[
	\Tr_{2^n/2^t}(x) = \sum_{k=0}^{n/t}x^{\left({2^t}\right)^k}. 
\]}

In \cite{budaghyan2009construction}, it is shown that for any non-zero $a \in \F_{2^{3m}}$, $m $ arbitrary, the DO polynomials
$$
f_1(x) = f'_1(x^3) = x^3+a^{-1}\Tr_{2^{3m}/2^3}(a^3x^9+a^6x^{18})
$$
and 
$$
f_2(x) = f'_2(x^3) = x^3+a^{-1}(\Tr_{2^{3m}/2^3}(a^3x^9+a^6x^{18}))^2
$$
define APN maps on $\F_{2^{3m}}$. Moreover, the maps $f'_1$ and $f'_2$ are bijective when  $m$ is even. 
For $m$ odd, the image sets of these maps contain $5\cdot 2^{3m-3}$  elements, as Propositions~\ref{prop:image-new1} and
\ref{prop:image-new2} show. 

\begin{proposition}\label{prop:image-new1}
Let $m$ be an odd integer and $a\in\F_{2^{3m}}^*$ be arbitrary.
Then the APN map $f:\F_{2^{3m}} \to \F_{2^{3m}}$ given by
\[
	f(x) = x^3+a^{-1}\Tr_{2^{3m}/2^3}(a^3x^9+a^6x^{18})
\]
satisfies $M_1(f) = 2^{3m-1}$, $M_4(f)=2^{3m-3}$. In particular $|\Image(f)|=5\cdot 2^{3m-3}$. 
\end{proposition}

\begin{proof}
We consider the equation $f(x)=f(y)$ on $\F_{2^{3m}}$. Since $x\mapsto x^3$ is a permutation on $\F_{2^n}$ with  $n$ odd, 
	it is sufficient to look at $f'(x)=f'(y)$, where
	\[ 
	f'(x) = x+a^{-1}\Tr_{2^{3m}/2^3}(a^3x^3+a^6x^6),
\]
	and  $f(x)=f'(x^3)$.
	Suppose $f'(x)=f'(y)$. Then
	\[
		x+a^{-1}\Tr_{2^{3m}/2^3}(a^3x^3+a^6x^6) = y+a^{-1}\Tr_{2^{3m}/2^3}(a^3y^3+a^6y^6)
	\]
	or equivalently,
	\begin{equation}\label{eq:x+y}
		\Tr_{2^{3m}/2^3}(a^3x^3+a^6x^6 + a^3y^3+a^6y^6) = a(x+y).
	\end{equation}
	In particular, $f'(x) = f'(y)$ only if $a(x+y) \in \F_8$.
	Let $z=x+y$. Taking the absolute trace on both sides of (\ref{eq:x+y}), we get
	\[
		\Tr_{2^3/2}(az)=\Tr_{2^{3m}/2}(a^3x^3+(a^3x^3)^2 + a^3(x+z)^3+(a^3(x+z)^3)^2) = 0.
	\]
	Let $\beta\in\F_8$ with $\beta^3=\beta+1$, then
	\[
		\Tr_{2^3/2}(\beta)=\Tr_{2^3/2}(\beta^2)=\Tr_{2^3/2}(\beta^4)=0,
	\]
	so that
	\[
		az\in\{0, \beta, \beta^2, \beta^4\}.
	\]
	If $az=0$ we have $z=0$ and $x=y$. So let $z=a^{-1}\beta^k$ with $k\in\{1,2,4\}$. Note that $x\mapsto x^k$ is a linear permutation on $\F_{2^{3m}}$. We have
	\begin{align*}
		& a^3x^3+a^6x^6+a^3(x+z)^3+a^6(x+z)^6 \\
		&=a^3x^3+a^6x^6+a^3(x^3+x^2z+xz^2+z^3)+a^6(x^6+x^4z^2+x^2z^4+z^6)\\
		&=a^3x^2z+a^3xz^2+a^3z^3+a^6x^4z^2+a^6x^2z^4+a^6z^6 \\
		&=a^2x^2\beta^k+ax\beta^{2k}+\beta^{3k}+a^4x^4\beta^{2k}+a^2x^2\beta^{4k}+\beta^{6k} \\
		&=(\beta^{3k}+\beta^{6k})+ax\beta^{2k}+a^2x^2(\beta^k+\beta^{4k})+a^4x^4\beta^{2k}.
	\end{align*}
	As $\beta^3=\beta+1$, we get $\beta^6=\beta^2+1$ and therefore $\beta^3+\beta^6=\beta^2+\beta=\beta^4$. Further $\beta+\beta^4=\beta^2$, so that
	\begin{equation}\label{eq2imageproof}
		 a^3x^3+a^6x^6+a^3(x+z)^3+a^6(x+z)^6
		=\beta^{4k}+\beta^{2k}(ax+(ax)^2+(ax)^4).
	\end{equation}
	We now need to ensure that \eqref{eq:x+y} holds. Using \eqref{eq2imageproof} and $m$ odd this turns into
	\begin{align*}
		\beta^k &=\Tr_{2^{3m}/2^3}(\beta^{4k}+\beta^{2k}(ax+(ax)^2+(ax)^4)) \\
						&=\beta^{4k}+\beta^{2k}\Tr_{2^{3m}/2^3}(ax+(ax)^2+(ax)^4) = \beta^{4k}+\beta^{2k}\Tr_{2^{3m}/2}(ax).
						 \\
						&=(\beta^4+\beta^2\Tr_{2^{3m}/2}(ax))^k
	\end{align*}
	Using again that $x\mapsto x^k$ is a permutation and that $\beta^4=\beta^2+\beta$ we obtain
	\[
		\beta^2(\Tr_{2^{3m}/2}(ax)+1)=0,
	\]
	which has a solution $x$ if and only if $\Tr_{2^{3m}/2}(ax)=1$.
	
	Concluding, we have $f'(x)=f'(x+z)$ if and only if $\Tr_{2^{3m}/2}(ax)=1$ and  $z\in\{0, a^{-1}\beta, a^{-1}\beta^2, a^{-1}\beta^4\}$. Since there are $2^{3m-1}$ elements $x$ with $\Tr_{2^{3m}/2}(ax)=1$, we get $M_4(f)=2^{3m-3}$. The map $f'$
	is injective on the hyperplane $\{x \in \F_{2^{3m}} ~|~ \Tr_{2^{3m}/2}(ax)=0 \}$, yielding $M_1(f)=2^{3m-1}$.
		\qed
\end{proof} 

The proof of next result is almost identical to the one of Proposition~\ref{prop:image-new1}: 

\begin{proposition} \label{prop:image-new2}
Let $m$ be an odd integer and $a\in\F_{2^{3m}}^*$ be arbitrary.
Then the APN map $f:\F_{2^{3m}} \to \F_{2^{3m}}$ given by 
\[
	f(x) = x^3+a^{-1}\Tr_{2^{3m}/2^3}(a^3x^9+a^6x^{18})^2
\]
satisfies $M_1(f) = 2^{3m-1}$, $M_4(f)=2^{3m-3}$. In particular $|\Image(f)|=5\cdot 2^{3m-3}$. 
\end{proposition}

\section{Relations between the image sets of APN maps and their Walsh spectrum}

Let $f \colon \F_{2^n} \rightarrow \F_{2^n}$. The Boolean fuctions $f_\lambda(x) = \Tr(\lambda f(x))$ with $\lambda \in \F_{2^n}^*$ are called the component functions of $f$. We call $f_\lambda$ a balanced component of $f$, if it takes the values $0$ and $1$ equally often, that is both $2^{n-1}$ times.
The Walsh transform of $f$ is defined by
	\[W_f(b,a)=\sum_{x \in \F_{2^n}}(-1)^{\Tr(bf(x)+ax)} \in \mathbb{Z},\]
	where $a,b \in \F_{2^n}, b\ne 0$.
	The multiset $\{*W_f(b,a) \colon b \in \F_{2^n}^*, a \in \F_{2^n}*\}$ is called the Walsh spectrum of $f$ and $\{*|W_f(b,a)| \colon b \in \F_{2^n}^*, a \in \F_{2^n}*\}$ is called the extended Walsh spectrum of $f$. 
	\begin{definition}
	Let $f \colon \F_{2^n} \rightarrow \F_{2^n}$, $n$ even. We say that the map $f$ has the classical Walsh spectrum if
	$$|W_f(b,a)| \in \{0,2^{n/2},2^{(n+2)/2}\}$$
	for any $b \in \F_{2^n}^*, a \in \F_{2^n}$ and the extended Walsh spectrum of $f$ contains 
	the values $0,2^{n/2},2^{(n+2)/2}$ precisely $(2^n-1)\cdot 2^{n-2}$-times, $(2/3)(2^n-1)(2^n)$-times and $(1/3)(2^n-1)(2^{n-2})$-times, respectively.
	\end{definition}
	Most of the known APN maps in even dimension have the classical Walsh spectrum, including the monomial APN maps with Gold and
Kasami exponents. There are  APN maps with non-classical 
Walsh spectra, see e.g. \cite{beierle-leander2020} for  such examples. 

\begin{definition}
	Let $f \colon \F_{2^n} \rightarrow \F_{2^n}$. A component function $f_\lambda$ is called plateaued with amplitude $t$  for an integer $t \geq 0$ if $W_f(\lambda,a) \in \{0,\pm 2^{\frac{n+t}{2}}\}$ for all $a \in \F_{2^n}$. If all component functions of $f$ are plateaued,  $f$ is called component-wise plateaued. 
	For $n$ odd the map $f$ is called almost bent if all its components $f_\lambda$ are plateaued with $t=1$.  For $n$ even, a plateaued component with $t=0$ is called a bent component.
\end{definition}

  It is well known that an almost bent map is necessarily  APN, and there are APN maps,
	which are not almost bent. 
Quadratic maps are always component-wise plateaued.  Also a crooked map, which is defined by the property that 
all its  differential sets are affine hyperplanes, is component-wise plateaued \cite{bending-crooked,kyureg-crooked}. Properties
of component functions of crooked maps are studied in \cite{charpin-crooked}.
Further examples of component-wise 
 plateaued maps can be found in \cite{carletplateaued}.\\

\noindent 
The next result gives a sufficient condition for a map to be APN in terms of its component functions.
\begin{proposition} [{\cite[Corollary 3]{berger2006almost}}]\label{prop:bentcomponents}
	Let $n$ be even and $f:\F_{2^n}\to \F_{2^n}$ be a component-wise plateaued map. If $f$ has $(2/3)(2^n-1)$ bent components and $(1/3)(2^n-1)$ components with amplitude $t=2$ then $f$ is APN.
\end{proposition}

The Parseval equation states
$$
\sum_{a\in\F_{2^n}}W_f(b,a)^2 =2^{2n}
$$ 
for any $b \in \F_{2^n}$. It implies that a component-wise plateaued map $f$ with $(2/3)(2^n-1)$ bent components and $(1/3)(2^n-1)$ plateaued components with amplitude $t=2$ has always the classical Walsh spectrum.\\

{
To prove the next theorem we use the following well known lemma.
\begin{lemma}\label{lem:gcds}
	Let $i,r \in \mathbb{N}$ be arbitrary. Then
	\begin{itemize}
		\item $\gcd(2^i-1,2^r+1) = \begin{cases} 2^{\gcd(i,r)} +1 & \text{if } i/\gcd(i,r) \text{ is even} \\
		1 & \text{else.} 
		\end{cases}$
		\item $\gcd(2^i+1,2^r+1) = \begin{cases} 2^{\gcd(i,r)} +1 & \text{if } i/\gcd(i,r) \text{ and } r/\gcd(i,r)\text{ are odd} \\
		1 & \text{else.} 
		\end{cases}$
	\end{itemize}
	\end{lemma} \qed
	The next theorem shows that almost-$(2^r+1)$-to-$1$ component-wise plateaued maps have a very special
	Walsh spectrum. The key step in its proof is the fact that  the components of such maps have weights
	divisible by ${2^r+1}$. This together with Lemma \ref{lem:gcds} and some basic identities for Walsh
	values allow to control the Walsh spectrum of $f$.} {Our proof is  an adaption of the one of Theorem 2
	from \cite{carlet2016quadratic}, where $(p^r+1)$-divisible quadratic maps of finite fields with
	an arbitrary characteristic $p$ are considered.}

{	
\begin{theorem}\label{thm:k-to-1}
	Let $n=2rm$ and  $f:\F_{2^n}\to \F_{2^n}$ be an almost-$(2^r+1)$-to-$1$ component-wise plateaued  map with $f(0)=0$ and $\omega(0)=1$, i.e. $0$ be the unique element with precisely one preimage. 
	 Then $f$ has $(2^r/(2^r+1))\cdot(2^n-1)$ bent components and $(2^n-1)/(2^r+1)$ components with amplitude $t=2r$. 
	Moreover, 
	\[W_f(b,0) \in \{(-1)^{m}2^{rm},(-1)^{m+1}2^{r(m+1)}\}\] for any $b \in \F_{2^n}^*$.
\end{theorem}
\begin{proof}
Let $b\in \F_{2^n}^*$ be arbitrary.
	Since $f$ is component-wise plateaued, $W_f(b,0)$  takes the values $0$ or $\pm 2^{rm+s}$ with $s\geq 0$. 
Note that since	$f$ is almost-$(2^r+1)$-to-$1$ with $f(0)=0$ and $\omega(0)=1$, the value 
$|\{x \in \F_{2^n}^* \colon \Tr(bf(x))=c\}|$ is divisible by $2^r+1$ for any $c \in \F_2$. Thus
\begin{eqnarray*}
	W_f(b,0) &=& |\{x \in \F_{2^n}^* \colon \Tr(bf(x))=0\}|-|\{x \in \F_{2^n}^* \colon \Tr(bf(x))=1\}|+1 \\
	 & \equiv & 1 \pmod {2^r+1}. 
	\end{eqnarray*}
		This shows in particular that $W_f(b,0) \neq 0$.
	Further, by Lemma~\ref{lem:gcds}, $2^{rm+s} \equiv 1 \pmod{2^r+1}$ if and only if $r|s$ and $m+(s/r)$ is even. Similarly, $-2^{rm+s} \equiv 1 \pmod{2^r+1}$ if and only if $r|s$ and $m+(s/r)$ is odd.
	Hence $W_f(b,0) =(-1)^{m+k}2^{r(m+k)}$ for a suitable $k\geq 0$.
	Define $$N_k=|\{b \in \F_{2^n}^* \colon |W_f(b,0)|=2^{r(m+k)}\}|$$ for an integer $k \geq 0$. 
	Since $f(x)=0$ holds only for $x=0$, we have 
	\[\sum_{b\in \F_{2^n}} W_f(b,0) = 2^n,\]
	which directly implies
		\[\sum_{b\in \F_{2^n}^*} W_f(b,0) = 0.\]
		Substituting  the possible values for $W_f(b,0)$ in the above equation, we get
		\begin{equation*}
			2^{rm}(N_0-2^{r}N_1+2^{2r}N_2-2^{3r}N_3+\dots )= 0,
		\end{equation*}
		implying
		\begin{equation} \label{eq:1}
			N_0-2^{r}N_1+2^{2r}N_2-2^{3r}N_3+\dots= 0.
		\end{equation}
		Now since $f$ is almost-$(2^r+1)$-to-$1$, for every fixed non-zero $x$ there are exactly $(2^r+1)$ elements $a \in \F_{2^n}$ satisfying
		$f(x)+f(x+a)=0$, and for $x=0$ only  $a=0$ solves it. Thus we get
		\begin{align*}
			\sum_{b \in \F_{2^n}} (W_f(b,0))^2 &=  \sum_{b \in \F_{2^n}} \sum_{x \in \F_{2^n}}\sum_{a \in \F_{2^n}} (-1)^{\Tr(b(f(x)+f(x+a)))} \\
			&=2^n((2^r+1) \cdot (2^n-1)+1)=2^{2n+r}+2^{2n}-2^{n+r}.
		\end{align*}
		In particular, 
		\begin{equation*}
		\sum_{b \in \F_{2^n}^*} (W_f(b,0))^2 =  2^{2n+r}-2^{n+r}.
		\end{equation*}
		Again, substituting  the possible values for $W_f(b,0)$, we get
		\begin{equation*}
			2^n(N_0+2^{2r}N_1+2^{4r}N_2+2^{6r}N_3+ \dots )= 2^{2n+r}-2^{n+r},
		\end{equation*}
		which immediately leads to
		\begin{equation} \label{eq:2}
			N_0+2^{2r}N_1+2^{4r}N_2+2^{6r}N_3+ \dots = 2^{n+r}-2^r.
		\end{equation}
		Clearly, we also have
		\begin{equation}\label{eq:3}
			N_0+N_1+N_2+\dots = 2^n-1.
		\end{equation}
		Adding Eq.~\eqref{eq:1} $(2^r-1)$-times to Eq.~\eqref{eq:2}, we get
		\begin{equation} \label{eq:4}
			2^rN_0+2^rN_1+(2^{2r}(2^r-1)+2^{4r})N_2+(2^{6r}-(2^r-1)2^{3r})N_3+\dots = 2^{n+r}-2^r.
		\end{equation}
		Observe that all coefficients in Eq.~\eqref{eq:4} are positive.
		Now, subtracting Eq.~\eqref{eq:3} $2^r$-times from Eq.~\eqref{eq:4} yields
		\begin{equation*}
			(2^{2r}(2^r-1)+2^{4r}-2^r)N_2+(2^{6r}-(2^r-1)2^{3r}-2^r)N_3+\dots = 0.
		\end{equation*}
		Here, all coefficients are again positive, so we conclude $N_2=N_3=\dots = 0$. From Eq.~\eqref{eq:1} and Eq.~\eqref{eq:3} we then immediately deduce that $N_0 =(2^r/(2^r+1))(2^n-1)$ and $N_1 = (2^n-1)/(2^r+1)$.
			\qed
\end{proof} }

Note that the conditions $f(0)=0$ and $\omega(0)=1$ are not restrictive when we consider  the extended Walsh spectrum: Indeed, otherwise we consider $f(x+c)+d$ with suitable  $c,d \in \F_{2^n}$, which is also component-wise plateaued and has  the same extended Walsh spectrum as $f$:
		\begin{align*}
	W_{f(x+c)+d}(b,a) &= \sum_{x \in \F_{2^n}}(-1)^{\Tr(b(f(x+c)+d)+ax)} \\
	&= (-1)^{\Tr(bd)}\sum_{x \in \F_{2^n}}(-1)^{\Tr(bf(x)+a(x+c))} \\
	& = (-1)^{\Tr(bd+ac)}W_{f}(b,a).
	\end{align*} \\

\noindent
{	
The two boundary cases $m=1$ and $r=1$ of Theorem~\ref{thm:k-to-1} imply
interesting extremal cases. For $m=1$, we get that a component-wise plateaued almost-$(2^{n/2}+1)$-to-$1$ map on $\F_{2^n}$ has $2^n-2^{n/2}$ bent components, which is the maximum number of bent components that a map on $\F_{2^n}$ can have~\cite{maxbent}. For $m=1$ the following result holds too:
\begin{proposition}
Let $r\in \mathbb{N}$, $n=2r$ and  $f:\F_{2^n}\rightarrow \F_{2^n}$ be an almost-$(2^r+1)$-to-$1$ map.
	 Then $f$ is component-wise plateaued if and only if it has $2^n-2^{n/2}$ bent components.
\end{proposition}
\begin{proof}
	One direction is covered by Theorem~\ref{thm:k-to-1}. Assume that $f$ has $2^n-2^{n/2}$ bent components. As mentioned before, we can assume without loss of generality that $f(0)=0$ and $\omega(0)=1$. By the proof of Theorem~\ref{thm:k-to-1} we have $W_f(b,0) \equiv 1 \pmod{2^{n/2}+1}$ for all $b \in \F_{2^n}^*$. Hence if $b$ defines
	a bent component, then $W_f(b,0)=-2^{n/2}$. Thus
	\begin{align*}
		0&=\sum_{b\in \F_{2^n}^*} W_f(b,0) \\
		 &=-(2^n-2^{n/2})2^{n/2}+\sum_{b \text{ not bent}} W_f(b,0),
	\end{align*}
	implying $\sum_{b \text{ not bent}} W_f(b,0)=2^{n}(2^{n/2}-1)$. The sum has $2^{n/2}-1$ terms and each term is less or equal to $2^n$, so necessarily it must hold $W_f(b,0)=2^n$ for every $b \in\F_{2^n}^*$ that does not  define a bent component. Then, by Parseval's equation, $W_f(b,a)=0$ for these $b$ and every $a \in \F_{2^n}^*$, so these components are also plateaued.
	\qed
\end{proof}}

The  case $r=1$ of Theorem~\ref{thm:k-to-1} shows that almost-3-to-1 component-wise plateaued  maps are APN and they have the classical Walsh spectrum:
\begin{theorem}\label{thm:walsh_spectrum_three_to_one}
	Let $n=2m$ and  $f:\F_{2^n}\to \F_{2^n}$ be an almost-$3$-to-$1$ component-wise plateaued map. 
	 Then $f$ is an APN map with the classical Walsh spectrum.
	Moreover, if $f(0)=0$ and $\omega(0)=1$, i.e. $0$ is the only element with precisely one preimage, then \[W_f(b,0) \in \{(-1)^{m}2^m,(-1)^{m+1}2^{m+1}\}\] for any $b \in \F_{2^n}^*$.
\end{theorem}
\begin{proof}
	The result follows from Theorem~\ref{thm:k-to-1} for $r=1$ and Proposition~\ref{prop:bentcomponents}.
			\qed
\end{proof}

\noindent
{In \cite[Corollary 10 and 11]{carletplateaued} it is proven that if $n$ is even, then
 all almost-3-to-1 plateaued maps of $\F_{2^n}$ have the same Walsh spectrum as the cube function $x\mapsto x^3$. This  is exactly the statement of Theorem \ref{thm:walsh_spectrum_three_to_one} too.
 The precise Walsh spectrum of the cube function is determined by Carlitz~\cite{carlitzgold} via
 a refined evaluation of certain Gauss sums.
The proof of Theorem \ref{thm:k-to-1} implies 
an elementary proof for  Carlitz's result on the value of the cubic exponential sum $S(b)=\sum_{x \in \mathbb{F}_{2^n}} (-1)^{\Tr(bx^3)}$. The latter is the key step for obtaining the  
Walsh spectrum of the cube function in  \cite{carlitzgold}.} \\

\noindent
It is well known that quadratic (not necessarily APN) maps as well as crooked maps are component-wise plateaued.
 Hence  Theorem~\ref{thm:walsh_spectrum_three_to_one} implies 
\begin{corollary}\label{cor:crooked-walsh}
Let $n=2m$ and  $f:\F_{2^n}\to \F_{2^n}$.
\begin{itemize}
\item[(a)] If $f$ is  almost-$3$-to-$1$ crooked map,
 then $f$ has  the classical Walsh spectrum. 
\item[(b)] If $f$ is almost-3-to-1 quadratic map, then it is APN with  the classical Walsh
spectrum.
\end{itemize}
\end{corollary}

Note that Corollaries~\ref{th_iff} and \ref{cor:crooked-walsh} confirm  Conjecture 1 stated in
\cite{villa2019apn},  that all APN maps of the form $f(x)=L_1(x^3)+L_2(x^9)$ in even dimension have the classical Walsh spectrum. Theorem~\ref{thm:walsh_spectrum_three_to_one} combined with Theorems~\ref{thm:zhoupott31} and \ref{apn-faruk}  yield that Zhou-Pott and G\"olo\u{g}lu APN maps have the classical Walsh spectrum. This is shown for Zhou-Pott and further  maps in \cite{anbar-walsh}, using Bezout's theorem on intersection points of two projective plane curves. Not all maps considered in \cite{anbar-walsh} are almost-3-to-1.\\

\noindent
By Corollary \ref{cor:crooked-walsh} the EA-class of a quadratic APN map with non-classical Walsh spectrum do
not contain an almost-3-to-1 map. The following related question is yet open:\\

\noindent
{{\bf Open Problem:} Let $n$ be even. Is there any APN DO map $f:\F_{2^n}\to \F_{2^n}$ with the classical Walsh spectrum, such that for any $\F_2$-linear map $f+l$ is not almost-$3$-to-$1$ (equivalently, that there is no
almost-$3$-to-$1$ map in the EA-class of $f$)\,?} \\

\noindent
Almost-$3$-to-$1$ APN maps with non-classical Walsh spectra exist; an example is the Dobbertin map $x \mapsto x^d$  on $\F_{2^n}$ where $10|n$, $n=5g$ and $d=2^{4g}+2^{3g}+2^{2g}+2^{g}-1$~\cite{ccddobbertin}.\\

\noindent
We conclude this section with some observations on the almost bent maps on $\F_ {2^n}$ with $n$ odd.
We use them in the next section to give an upper bound on the image set of such maps.
Next lemma describes  a direct connection between $N(f)$ and the number $N_0$ of balanced component functions
of almost bent maps.
\begin{lemma} \label{lem:AB}
	Let $n$ be odd and $f:\F_{2^n} \to \F_{2^n}$ be  almost bent. 
	Set 
	\begin{align*}
	N_0&=|\{b \in \F_{2^n}^* \colon W_f(b,0)=0\}|\\
	N_+ &= |\{b \in \F_{2^n}^* \colon W_f(b,0)=2^{(n+1)/2}\}|\\
	N_- &= |\{b \in \F_{2^n}^* \colon W_f(b,0)=-2^{(n+1)/2}\}|.
	\end{align*}
	Then these  three values are determined by $N(f)$ in the following way:
	\begin{align*}
	N_0&=2^n-1+2^{n-1}-N(f)/2\\
	N_+ &= N(f)/4 -2^{n-2}+2^{(n-3)/2}(\omega (0)-1)\\
	N_- &= N(f)/4 -2^{n-2}-2^{(n-3)/2}(\omega (0)-1).
	\end{align*}
\end{lemma}
\begin{proof}
	Clearly, we have
	\begin{equation} \label{eq:ab_1}
		N_0+N_++N_- = 2^n-1.
	\end{equation}
	Further, we have 
		\begin{align*}
			\sum_{b \in \F_{2^n}} (W_f(b,0))^2 &=  \sum_{b \in \F_{2^n}} \sum_{x \in \F_{2^n}}\sum_{y \in \F_{2^n}} (-1)^{\Tr(b(f(x)+f(y)))} \\
			&=2^n N(f),
		\end{align*}
	which implies
			\begin{equation*}
		\sum_{b \in \F_{2^n}^*} (W_f(b,0))^2 =  2^n N(f) - 2^{2n}.
		\end{equation*}
	Rewriting this equation yields
	\begin{equation*}
		2^{n+1}(N_++N_-) = 2^n N(f) - 2^{2n}
	\end{equation*}
	or, equivalently, 
		\begin{equation}
		N_++N_-=  N(f)/2 - 2^{n-1}.
	\label{eq:ab_2}
	\end{equation}
	Moreover, we have
		\[\sum_{b\in \F_{2^n}} W_f(b,0) = 2^n\cdot \omega (0),\]
	which directly implies
		\[\sum_{b\in \F_{2^n}^*} W_f(b,0) = 2^n\cdot (\omega (0)-1),\]
		which yields
	\begin{equation*}
		2^{(n+1)/2}(N_+-N_-) = 2^n\cdot (\omega (0)-1),
	\end{equation*}	
	and
	\begin{equation}\label{eq:ab_3}
		N_+-N_- = 2^{(n-1)/2}\cdot (\omega (0)-1).
	\end{equation}	
	Subtracting Eq.~\eqref{eq:ab_2} from Eq.~\eqref{eq:ab_1} yields
	\begin{equation*}
		N_0 = 2^n+2^{n-1}-N(f)/2-1.
	\end{equation*}
	Similarly adding Eq.~\eqref{eq:ab_2} and Eq.~\eqref{eq:ab_3} we get that $N(f)$ must be divisible by $4$ and that
	\begin{equation*}
		N_+ = N(f)/4-2^{n-2}+2^{(n-3)/2}(\omega (0)-1).
	\end{equation*}
	The value of $N_-$ then follows immediately from Eq.~\eqref{eq:ab_1}. 
		\qed
\end{proof}

Lemma~\ref{lem:AB} directly implies

\begin{corollary}\label{cor:AB}
Let $n$ be odd and  $f:\F_{2^n} \to \F_{2^n}$ be almost bent. Then
\begin{itemize}
\item[(a)]
$N(f)$ is  divisible by $4$.
\item[(b)] The number of balanced component functions of $f$ is odd.
In particular, every almost bent function has at least one balanced component function.
\item[(c)]	$N(f) \leq 3\cdot 2^n-4$ and $f$ is not  zero-difference 2-balanced.
\item[(d)] {\[|\Image(f)| > \frac{2^n+1}{3}.\]}
	\end{itemize}
\end{corollary}
\begin{proof}
Statement (a) holds since $N_+$ and $N_-$ in Lemma~\ref{lem:AB} are integers. Then (b) is a direct consequence of (a) and
 Lemma~\ref{lem:AB}. Using Corollary~\ref{cor:Nf} and (a), we get that 
$N(f) \leq 3\cdot 2^n-4$ and hence $f$ is not  zero-difference 2-balanced.
{Theorem \ref{thm:APN_Lower_Bound} with (c) imply (d), since if the lower bound is fulfilled then necessarily $N(f)=3\cdot 2^n-2$. }
\qed
\end{proof}

\begin{remark}
Recall that any crooked map is almost bent if $n$ is odd. Property (c) in Corollary~\ref{cor:AB} implies that at least one differential set of a crooked map on $\F_{2^n}$ with $n$ odd
is a complement of a hyperplane. Equivalently, for $n$ odd there is no crooked
map such that all its  difference sets are hyperplanes.
To the contrary if $n$ is odd then there are bijective crooked maps, for which necessarily all
differential sets are complements of hyperplanes. Interestingly, this property is the other way around if $n$ is even.
Then  crooked maps, for which  all  differential sets are hyperplanes, do exist (for instance, $x\mapsto x^3$ as observed in \cite{kyureg-crooked}). But there are no crooked
maps with all their differential sets being complements of hyperplanes. The latter is a consequence of the 
non-existence of bijective crooked maps in even dimension \cite{kyureg-crooked}. 
\end{remark}

\section{Upper bounds on the image sets of APN maps}

In previous sections  we  used the value $N(f)$ to obtain a lower bound for the image size of some special maps. In~\cite{coultersenger}, information on $N(f)$ was used to prove an \emph{upper} bound on the image size of maps, significantly for  planar maps. 

\begin{theorem}[{\cite[Theorem 2]{coultersenger}}] \label{thm:upperbound}
	Let $f\colon \F_{2^n} \rightarrow \F_{2^n}$. Then
	\[|\Image(f)| \leq 2^n-\frac{2N(f)-2^{n+1}}{1+\sqrt{4N(f)-2^{n+2}+1}}=2^n-\frac{1}{2}(\sqrt{4N(f)-2^{n+2}+1}-1).  \]
\end{theorem} \qed

\medskip

\noindent
The equality 
\[\frac{2N(f)-2^{n+1}}{1+\sqrt{4N(f)-2^{n+2}+1}} = \frac{1}{2}(\sqrt{4N(f)-2^{n+2}+1}-1)\]
is not mentioned in~\cite{coultersenger}, but it can be verified easily by expanding the fraction with $1-\sqrt{4N(f)-2^{n+2}+1}$.\\

If $n$ is even, Theorem \ref{thm:upperbound} implies in particular an upper bound on the image size of a map  depending on the number of  its bent components as observed in Theorem \ref{thm:bentupperbound}. For an  odd $n$
Lemma~\ref{lem:AB}  yields an upper bound for the image size of almost bent maps. Since almost bent maps, contrary to the planar ones, can  be permutations, the bound is more involved.

\begin{theorem} \label{thm:ABupperbound}
	Let $f \colon \F_{2^n} \rightarrow \F_{2^n}$ be an almost bent map. Set $k = \max\{\omega(a) | a\in\F_{2^n}\}$, i.e. there exists an element $c \in \F_{2^n}$ with $k$ preimages under $f$ and there is no element with more than $k$ preimages. 
	Then
	\[|\Image(f)|  \leq 2^n - \frac{k-1}{k}2^{(n+1)/2}.\]
	In particular, if $f$ is not a permutation, then
	\begin{equation}
			|\Image(f)|  \leq 2^n - 2^{(n-1)/2}.
	\label{eq:ABupperbound}
	\end{equation}

\end{theorem}
\begin{proof}
	By Eq.s~\eqref{eq_rMr} and \eqref{eq_r2Mr}
		\[N(f)-2^n = \sum_{r=1}^k r(r-1)M_r \leq k \sum_{r=1}^k (r-1)M_r,\]
	implying
	\begin{equation}
		\sum_{r=1}^k (r-1)M_r \geq \frac{N(f)-2^n}{k}.
	\label{eq:imagesizeupper}
	\end{equation}
	Set $f'(x) = f(x) - c$ with $\omega(c)=k$. Clearly, $f'$ is also almost bent and it satisfies $N(f) = N(f')$ and $|\Image(f)| = |\Image(f')|$, and additionally for $f'$ we have $\omega(0) =k$. We apply Lemma~\ref{lem:AB} to $f'$. Then
	\[0 \leq N_- = N(f)/4 -2^{n-2}-(k-1)2^{(n-3)/2}, \]
	which leads to
	\[N(f) - 2^n \geq (k-1)2^{(n+1)/2}.\] 
	Then, using Eq.~\eqref{eq:imagesizeupper},
	\begin{align*}
		|\Image(f)| &= \sum_{r=1}^k M_r = \sum_{r=1}^k r M_r - \sum_{r=1}^k (r-1) M_r \\
		&= 2^n - \sum_{r=1}^k (r-1) M_r \leq 2^n- \frac{N(f)-2^n}{k} \leq 2^n - \frac{k-1}{k}2^{(n+1)/2}.
	\end{align*}

	If $f$ is not a permutation, then $k>1$ and $\frac{k-1}{k}\geq 1/2$, completing the proof.
\qed
\end{proof}

\begin{remark}
\begin{enumerate}
	\item From Theorem~\ref{thm:ABupperbound}, it is clear that almost bent maps that satisfy the bound  in \eqref{eq:ABupperbound} with equality must satisfy $\max\{\omega(a) | a\in\F_{2^n}\}=2$. For such a map we then necessarily have $M_1(f)= 2^n-2^{(n+1)/2}$ and $M_2(f)=2^{(n-1)/2}$. However we believe that the bound is not sharp.
	\item {The bound of Theorem~\ref{thm:ABupperbound} is similar in style to the well-known general upper bound on the image size of  maps by Wan~\cite{wanbound}, stating that if  $f \colon \F_{2^n} \rightarrow \F_{2^n}$ is not bijective then 
	\[|\Image(f)| \leq 2^n-\frac{2^n-1}{d},\]
	where $d$ is the degree of $f$. Another bound similar to Wan's bound  
	is: If  $f \colon \F_{2^n} \rightarrow \F_{2^n}$ is not bijective 
	and has index $l$ then 
	\[|\Image(f)| \leq 2^n-\frac{2^n-1}{l},\]
	see \cite{wang-value-set} for more details and the definition of the index of maps. For almost bent maps with known small degree or index, these upper bounds are stronger
	than the one in Theorem~\ref{thm:ABupperbound}.}
\end{enumerate}
\end{remark}

In even dimension, it is well known that maps with bent component functions cannot be permutations since bent functions are never balanced. Using Theorem~\ref{thm:upperbound}, we present an upper bound on the image size of a map depending on the number of bent component functions. This also yields an upper bound for the image size of component-wise plateaued APN maps in even dimension since such maps have always many bent component functions.

\begin{theorem} \label{thm:bentupperbound}
	Let $n$ be even and $f\colon \F_{2^n} \rightarrow \F_{2^n}$ be a map with $t$ bent component functions. Then $N(f) \geq t+2^n$ and
		\[|\Image(f)| \leq 2^n-\frac{1}{2}(\sqrt{4t+1}-1).  \]
\end{theorem}
\begin{proof}
	We use again the relation
	\[2^n N(f) = \sum_{b \in \F_{2^n}} W_f(b,0)^2 = 2^{2n}+\sum_{b \in \F_{2^n}^*} W_f(b,0)^2.\]
	If $x \mapsto \Tr(bf(x))$ is bent, then $W_f(b,0)^2 = 2^n$, so 
		\[2^nN(f) \geq 2^{2n}+t \cdot 2^n,\]
	implying $N(f) \geq t +2^n$. 
	The remaining  follows from Theorem~\ref{thm:upperbound}. 
	\qed
\end{proof}

Theorems~\ref{thm:ABupperbound} and~\ref{thm:bentupperbound} yield an upper bound on the image size for  component-wise plateaued APN maps.

\begin{theorem}
Let $f:\F_{2^n} \to \F_{2^n}$ be a component-wise plateaued APN map, and non-bijective if $n$ is odd. Then
			\[|\Image(f)| \leq 
			\begin{cases}
		   2^n - 2^{(n-1)/2}  & n \text{ is odd,} \\
			2^n-\frac{1}{2}(\sqrt{\frac{8}{3}(2^n-1)+1}-1)< 2^n-\sqrt{\frac{2}{3}(2^n-1)}+1/2 & n \text{ is even}.
		\end{cases}
			  \]
\end{theorem}
\begin{proof}
The statement for $n$ odd follows from Theorems~\ref{thm:ABupperbound}, since every component-wise plateaued APN map 
is almost bent.
The upper bound for $n$ even is a direct consequence from Theorem \ref{thm:bentupperbound} and the fact that a component-wise plateaued APN map has at least $(2/3)(2^n-1)$ bent component functions \cite[Corollary 3]{berger2006almost}.
	\qed
\end{proof}

\begin{acknowledgement} {
We thank our colleagues for all the comments which help us to improve the presentation of this paper.
Our special thank is to Zeying Wang for pointing us an inaccuracy in the earlier version of Theorem \ref{th:2.7}  and consequently in Corollary \ref{cor:apn-1-2}.
 We thank Steven Wang for bringing to our
attention references \cite{wanbound,wang-value-set}. }

\end{acknowledgement}


\begin{thebibliography}{100}

\bibitem{anbar-walsh}
N.~Anbar, T.~Kalayci and W.~Meidl.
\newblock{Determining the Walsh spectra of Taniguchi's and related APN-functions.}
 	\newblock{\em Finite Fields and Their Applications}, 60, Art. 101577, 2019.

\bibitem{bending-crooked}
T.D.~Bending and D.~Fon-Der-Flaass.
\newblock Crooked functions, bent functions, and distance regular graphs.
\newblock{\em Electron. J. Combin.}, 5(1), p.R34, 1998.

\bibitem{beierle-leander2020}
C.~Beierle and G.~Leander.
\newblock New instances of quadratic {APN} functions.
\newblock  	arXiv:2009.07204, 2020. 

\bibitem{berger2006almost}
T.~P. Berger, A.~Canteaut, P.~Charpin, and Y.~Laigle-Chapuy.
\newblock On almost perfect nonlinear functions over $\mathbb{F}_2^n$.
\newblock {\em IEEE Transactions on Information Theory}, 52(9):4160--4170,
  2006.


\bibitem{nybergsurvey}
C.~Blondeau and K.~Nyberg.
\newblock Perfect nonlinear functions and cryptography.
\newblock {\em Finite Fields and Their Applications}, 32:120 -- 147, 2015.



\bibitem{budaghyan2009constructing}
L.~Budaghyan, C.~Carlet, and G.~Leander.
\newblock Constructing new {APN} functions from known ones.
\newblock {\em Finite Fields and Their Applications}, 15(2):150--159, 2009.

\bibitem{budaghyan2009construction}
L.~Budaghyan, C.~Carlet, and G.~Leander.
\newblock On a construction of quadratic {APN} functions.
\newblock In {\em 2009 IEEE Information Theory Workshop}, pages 374--378. IEEE,
  2009.

\bibitem{ccddobbertin}
A.~Canteaut, P.~Charpin, and H.~Dobbertin.
\newblock Weight divisibility of cyclic codes, highly nonlinear functions on
  $\mathbb{F}_{2^m}$, and crosscorrelation of maximum-length sequences.
\newblock {\em SIAM Journal on Discrete Mathematics}, 13(1):105--138, 2000.

\bibitem{carletplateaued}
C.~{Carlet}.
\newblock Boolean and vectorial plateaued functions and {APN} functions.
\newblock {\em IEEE Transactions on Information Theory}, 61(11):6272--6289,
  2015.

\bibitem{cczpaper}
C.~Carlet, P.~Charpin, and V.~Zinoviev.
\newblock Codes, bent functions and permutations suitable for {DES}-like
  cryptosystems.
\newblock {\em Designs, Codes and Cryptography}, 15(2):125--156,  1998.

\bibitem{carletdingcodes}
C.~{Carlet}, {C.~Ding}, and {J.~Yuan}.
\newblock Linear codes from perfect nonlinear mappings and their secret sharing
  schemes.
\newblock {\em IEEE Transactions on Information Theory}, 51(6):2089--2102,
  2005.

\bibitem{carletdingnonlinearities}
C.~Carlet and C.~Ding.
\newblock Nonlinearities of {S}-boxes.
\newblock {\em Finite Fields and Their Applications}, 13(1):121 -- 135, 2007.

\bibitem{carlet2016quadratic}
C.~Carlet, G.~Gong, and Y.~Tan.
\newblock Quadratic zero-difference balanced functions, {APN} functions and
  strongly regular graphs.
\newblock {\em Designs, Codes and Cryptography}, 78(3):629--654, 2016.

\bibitem{carlet-heuser-picek}
C. Carlet, A. Heuser and S. Picek. 
\newblock Trade-Offs for S-Boxes: Cryptographic Properties and Side-Channel Resilience. 
\newblock{\em Proceedings of ACNS 2017}, Lecture Notes in Computer Science 10355,  393-414, 2017.

\bibitem{carlet2020}
C.~Carlet.
\newblock Bounds on the nonlinearity of differentially uniform functions by means of their image set size, and on their distance to affine functions.
\newblock {\em  Cryptology ePrint Archive}, Report 2020/1529. 



\bibitem{carlitzgold}
L.~Carlitz.
\newblock Explicit evaluation of certain exponential sums.
\newblock{\em Math. Scand.},  44 (1979), 5--16.

\bibitem{charpin-crooked}
P.~Charpin.
\newblock Crooked functions.
\newblock In { \em Finite Fields and Their Applications}, ed. J.Davis, 87--102,
  2020.

\bibitem{charpin-kyureg-ffa}
P.~Charpin and G.~Kyureghyan.
\newblock When does $G(x) +Tr(H(x))$ permute $\mathbb{F}_{p^n}$.
\newblock {\em Finite Fields and Their Applications}, 15(5):615--632, 2009.

\bibitem{charpin2008class}
P.~Charpin and G.~M. Kyureghyan.
\newblock On a class of permutation polynomials over $\mathbb{F}_{2^n}$.
\newblock In {\em International Conference on Sequences and Their
  Applications - SETA 2008}, Lecture Notes in Comput. Sci. 5203, Springer, 368--376, 2008.

\bibitem{coulter2011number}
R.~S. Coulter and R.~W. Matthews.
\newblock On the number of distinct values of a class of functions over a
  finite field.
\newblock {\em Finite Fields and Their Applications}, 17(3):220--224, 2011.

\bibitem{coultersenger}
R.~S. Coulter and S.~Senger.
\newblock On the number of distinct values of a class of functions with finite
  domain.
\newblock {\em Annals of Combinatorics}, 18(2):233--243,  2014.

\bibitem{czerwinski2020minimal}
I.~Czerwinski.
\newblock On the minimal value set size of {APN} functions.
\newblock {\em Cryptology ePrint Archive,} Report 2020/705.

\bibitem{ding2006diffset}
C.~Ding and Y.~Jin.
\newblock A family of skew Hadamard difference sets.
\newblock {\em Journal of Combinatorial Theory A}, 113(7):1526--1535, 2006.

\bibitem{farukapn}
F.~G\"olo\u{g}lu.
\newblock Gold-hybrid APN functions.
\newblock Preprint, 2020. \\
see also  https://boolean.w.uib.no/files/2020/09/gologlu$\_$slides.pdf


\bibitem{kaspers}
Ch.~Kaspers and Yue Zhou.
\newblock
A lower bound on the number of
inequivalent APN functions.
\newblock arXiv:2002.00673v2, 2020. 

\bibitem{kyureg-mueller-wang}
G.~M. Kyureghyan, P.~M\"uller and Q. Wang.
\newblock On the size of Kakeya sets in finite vector spaces.
\newblock {\em The Elec. Journal of Combinatorics }, 20(3), P36, 2013.



\bibitem{kyureg-crooked}
G.~M. Kyureghyan.
\newblock Crooked maps in $\mathbb{F}_{2^n}$.
\newblock {\em Finite Fields and Their Applications}, 13(3):713 -- 726, 2007.

\bibitem{kyureghyan2008some}
G.~M. Kyureghyan and A.~Pott.
\newblock Some theorems on planar mappings.
\newblock In Proceedings of {\em International Workshop on the Arithmetic of Finite Fields}, Lecture Notes in Comput. Sci., 5130,
  117--122, 2008.
	
\bibitem{maxbent}
A. Pott, E. Pasalic, A. Muratovic-Ribic and S. Bajric.
\newblock On the maximum number of bent components of vectorial functions.
\newblock In {\em IEEE Transactions on Information Theory}, 64(1):403--411, 2018.

\bibitem{pottsurvey}
A.~Pott.
\newblock Almost perfect and planar functions.
\newblock {\em Designs, Codes and Cryptography}, 78(1):141--195,  2016.

\bibitem{villa2019apn}
I.~Villa.
\newblock On {APN} functions ${L}_1(x^3)+{ L}_2(x^9)$ with linear ${L}_1$ and
  ${L}_2$.
\newblock {\em Cryptography and Communications}, 11(1):3--20, 2019.


\bibitem{wanbound}
 D.~Wan.
\newblock A {$p$}-adic lifting lemma and its applications to permutation
              polynomials
\newblock In {\em Finite fields, coding theory, and advances in communications
              and computing}, Lecture Notes in Pure and Appl. Math., 141:209--216, 1993.


\bibitem{wang-value-set}	
S.~Wang.						
\newblock Polynomials over finite fields: an index approach.
 \newblock in {\em  Combinatorics and Finite Fields. Difference Sets, Polynomials, Pseudorandomness and Applications}, Degruyter, 319--348, 2019.							

\bibitem{weng2007pseudo}
G.~Weng, W.~Qiu, Z.~Wang, and Q.~Xiang.
\newblock Pseudo-paley graphs and skew Hadamard difference sets from
  presemifields.
\newblock {\em Designs, Codes and Cryptography}, 44(1-3):49--62, 2007.

\bibitem{weng2012further}
G.~Weng and X.~Zeng.
\newblock Further results on planar {DO} functions and commutative semifields.
\newblock {\em Designs, Codes and Cryptography}, 63(3):413--423, 2012.

\bibitem{zhou2013new}
Y.~Zhou and A.~Pott.
\newblock A new family of semifields with 2 parameters.
\newblock {\em Advances in Mathematics}, 234:43--60, 2013.

\end{thebibliography}
\end{document}